\documentclass[12pt]{amsart}

\usepackage[dvips]{graphicx}
\usepackage{geometry}
\usepackage{a4}
\usepackage{calc}
\usepackage{lipsum}

\usepackage[hang,small,bf]{caption}
\usepackage[subrefformat=parens]{subcaption}
\makeatletter
\renewcommand\subsection{\@startsection{subsection}{2}%
  \z@{-.5\linespacing\@plus-.7\linespacing}{.5\linespacing}%
  {\normalfont\scshape}}
\renewcommand\subsubsection{\@startsection{subsubsection}{3}%
  \z@{.5\linespacing\@plus.7\linespacing}{-.5em}%
  {\normalfont\scshape}}
\makeatother

  \makeatletter
  \@addtoreset{equation}{section} 
  \makeatother

\usepackage{color}
\usepackage[mathscr]{eucal}
\usepackage{amsmath, amssymb}
\usepackage{amsthm}
\usepackage{mathrsfs}
\usepackage{enumerate}
\usepackage{amsmath,amssymb}
\usepackage{bm}
\usepackage{graphicx}
\usepackage{verbatim}
\usepackage{wrapfig}
\usepackage{ascmac}
\usepackage{multicol}

\newtheorem{prop}{Proposition}[section]
\newtheorem{theo}[prop]{Theorem}

\newtheorem{lemm}[prop]{Lemma}

\newtheorem{cor}[prop]{Corollary}

\theoremstyle{definition}  
\newtheorem{defi}[prop]{Definition}

\theoremstyle{remark}
\newtheorem{rem}[prop]{Remark}

\newcommand{\R}{{\mathbb R}}

\newcommand{\relmiddle}[1]{\mathrel{}\middle#1\mathrel{}}

\title[A maximal element of a moduli space of Riemannian metrics]{A maximal element of a moduli space of Riemannian metrics}

\author{
  Yuichiro Taketomi
}

\thanks{This work was supported by the Research Institute for Mathematical Sciences, an International Joint Usage/Research Center located in Kyoto University.
  This work was partly supported by Osaka Central Advanced Mathematical
Institute: MEXT Joint Usage/Research Center on Mathematics and Theoretical
Physics JPMXP0619217849. }

\begin{document}

\maketitle

\begin{abstract}
  For a given smooth manifold, we consider the moduli space of Riemannian metrics up to isometry and scaling.
One can define a preorder on the moduli space by the size of isometry groups.
We call a Riemannian metric that attains a maximal element with respect to the preorder a maximal metric.
Maximal metrics give nice examples of self-similar solutions for various metric evolution equations such as the Ricci flow.
In this paper, we construct many examples of maximal metrics on Euclidean spaces.
\end{abstract}

\section{Introduction}

Let $X$ be a connected smooth manifold.
Denote by $\mathfrak{M}(X)$ the set of all smooth Riemannian metrics on $X$.
Define an equivalent relation $\sim$ on $\mathfrak{M}(X)$ as follows:
there exists $\lambda >0$ such that $(X,\lambda g)$ and $(X,h)$ are isometric with each other.
Denote by $\mathfrak{M}(X)/_{\! \sim}$ the quotient space with respect to the equivalent relation $\sim$.
Note that the moduli space $\mathfrak{M}(X)/_{\! \sim}$ can be understood as the orbit space $(\R_{>0}\times \mathrm{Diff}(X))\backslash \mathfrak{M}(X)$.

The moduli space $\mathfrak{M}(X)/_{\! \sim}$ has often been 
considered to understand ``nice'' Riemannian metrics.
For examples, 
the normalized total scalar curvature
\begin{equation*}
  \tilde{\mathcal{S}} : \mathfrak{M}(X) \to \R, \quad g \mapsto \mathrm{vol}(X,g)^{\frac{2}{n}-1}\int_{X}\mathrm{scal}_{g}dV_{g}
\end{equation*}
has been studied actively for a compact manifold $X$.
Here, $\mathrm{vol}(X,g)$ is the Riemannian volume,
$\mathrm{scal}_{g}$ is the scalar curvature, and $dV_{g}$ is the Riemannian volume element.
For examples, Einstein metrics on a compact manifold $X$ are characterized by critical points of $\tilde{\mathcal{S}} : \mathfrak{M}(X) \to \R$.
Since the normalized total scalar curvature $\tilde{\mathcal{S}}$ is invariant under the action of $\R_{>0} \times \mathrm{Diff}(X)$ on $\mathfrak{M}(X)$,
$\tilde{\mathcal{S}}$ can be regarded as the function on the moduli space $\mathfrak{M}(X)/_{\! \sim}$.
Another important example is the Ricci flow
\begin{equation}
  \label{eq:0712133631}
    \frac{\partial}{\partial t}g_{t} = -2 \mathrm{Ric}_{g_{t}},
\end{equation}
where $\mathrm{Ric}_{g}$ is the Ricci tensor for a Riemannian metric $g$.
Since the Ricci tensor $\mathrm{Ric}_{g}$ is invariant under the scaling (\textit{i.e.} $\mathrm{Ric}_{cg} = \mathrm{Ric}_{g}$)
and diffeomorphisms (\textit{i.e.} $\mathrm{Ric}_{\varphi^{\ast}g} = \varphi^{\ast}\mathrm{Ric}_{g}$, where $\ast$ means the pullback),
the Ricci flow equation can be regarded as a flow on the moduli space $\mathfrak{M}(X)/_{\! \sim}$.
Then, for examples,  self-similar solutions of the Ricci flow are considered as stationary solutions of the corresponding flow on $\mathfrak{M}(X)/_{\! \sim}$.

Many ``nice'' metrics (\textit{e.g.} Einstein metrics, Ricci soliton) can be understood as distinguished points
on the moduli space $\mathfrak{M}(X)/_{\! \sim}$ with respect to some curvature condition.
In this paper, we introduce a class of nice Riemannian metrics, which is defined as
a special point on the moduli spaces $\mathfrak{M}(X)/_{\! \sim}$ with respect to the size of isometry groups. 
For $[g], [h] \in \mathfrak{M}(X)/_{\! \sim}$,
denote by $[g] \prec [h]$ if $\mathrm{Isom}(X,g) \subset \mathrm{Isom}(X,h^{\prime})$ for some $h^{\prime} \in [h]$.
Then $\prec$ defines an order on $\mathfrak{M}(X)/_{\! \sim}$.
Note that $\prec$ is not a partial order but a preorder.
That is,
$\prec$ satisfies reflexivity and transitivity, however, does not satisfy asymmetry.
\begin{defi}
 \label{defi:1207154844}
  We call a metric $g \in \mathfrak{M}(X)$ \textit{maximal} if the equivalent class $[g] \in \mathfrak{M}(X)/_{\! \sim}$ is a maximal element
  with respect to the preorder $\prec$ (\textit{i.e.} $[g] \prec [h]$ implies $[g] = [h]$). 
\end{defi}
Roughly speaking,
the order $\prec$ defines a barometer of the excellence of Riemannian metrics via the size of isometry groups,
and maximal metrics are ``maximally nice'' metrics with respect to this  barometer.

Note that
$g \in \mathfrak{M}(X)$ is a maximal metric if and only if
$g$ satisfies the following condition : 
\begin{equation}
  \label{eq:0114133517}
  \mathrm{Isom}(X,g) \subset \mathrm{Isom}(X,h) \ \Rightarrow \ [g] = [h]\quad (h \in \mathfrak{M}(X)).
\end{equation}

An important example of maximal metrics is an isotropy irreducible metric.
For more details, see Subsection~\ref{20220930164546}.

One of the nice motivation to study maximal metrics is that
$g$ give examples of self-similar solutions for various metric evolution equations
\begin{equation*}
  \frac{\partial}{\partial t} g_{t} = \mathcal{R}(g_{t})  \quad (g_{t} \in \mathfrak{M}(X)),
\end{equation*}
where $\mathcal{R}$ is a map from $\mathfrak{M}(X)$ to the set of all symmetric $(0,2)$-tensors $\mathfrak{S}(X)$.
We show that homogeneous maximal metrics are soliton for various metric evolution equations. For examples,
\begin{prop}
  \label{prop:0526134932}
A homogeneous maximal metric $g$ is a Ricci soliton.
\end{prop}
For more details, see Subsection~\ref{20220930164756}.

\begin{rem}
  \label{rem:0908144252}
Homogeneous Ricci solitons have been studied actively.
Recently, a proof of the Alekseevskii conjecture
has been announced by Lafuente and B\"{o}hm (\cite{BL2021}).
Also, Jablonski has shown that the Alekseevskii conjecture is equivalent to the generalized Alekseevskii conjecture which asserts that
only Euclidean spaces can admit homogeneous expanding Ricci solitons (\cite{MR3268781}).
Note that a shrinking homogeneous Ricci soliton manifold is the Riemannian product of a compact Einstein manifold and a Euclidean space (\cite{MR2673425, MR2480290}),
and a steady homogeneous Ricci soliton is the Riemannian product of a compact flat manifold and a Euclidean space (\cite{MR3268781, MR2480290}).
These and the generalized Alekseevskii conjecture conclude that a noncompact homogeneous Ricci soliton irreducible Riemannian manifold is diffeomorphic to a Euclidean space.
Therefore, essential homogeneous maximal metrics on noncompact manifolds can only exist on Euclidean spaces.

For the compact case, one can show that maximal metrics on compact manifolds must be isotropy irreducible.
We will give a proof in forthcoming paper.
\end{rem}

Another important property of maximal metrics is that they have maximal isometry groups
in the sense that
$\mathrm{Isom}(X,\langle , \rangle) \subset \mathrm{Isom}(X,\langle , \rangle^{\prime})$ implies
$\mathrm{Isom}(X,\langle , \rangle) = \mathrm{Isom}(X,\langle , \rangle^{\prime})$ for all Riemannian metric $\langle , \rangle^{\prime}$ on $X$.
For a maximal metric $g \in \mathfrak{M}(X)$, if the number of connected components of $\mathrm{Isom}(X,g)$ is finite
then $g$ has a maximal isometry group.
In particular,
\begin{prop}
  \label{prop:0526135154}
A homogeneous maximal metric $g$ has a maximal isometry group.
\end{prop}
For more details, see Subsection~\ref{20220922174711}.

A similar notion of maximal metric has been introduced by Jablonski and Gordon for left-invariant metrics, which is called maximal symmetry (\cite{MR3909903}).
The relationship between maximal metrics and maximal symmetry metrics will be discussed in Subsection~\ref{20220922174639} and Subsection~\ref{20220922174711}, and is summarized in
Figure~\ref{figure:0712142357} and Figure~\ref{figure:0712142537}.

A goal of this paper is to construct various examples of 
maximal metrics on Euclidean spaces which are not isotropy irreducible.
Our strategy to construct examples is to study a moduli space of left-invariant metrics on a Lie group $G$, which is given as
the orbit space of the action of $\R_{>0} \mathrm{Aut}(\mathfrak{g})$ on the set of all inner products $\mathfrak{m}(\mathfrak{g})$ on $\mathfrak{g} = \mathrm{Lie}(G)$.
As a preparation, in Section~3, we study some general theory of isolated orbits for an isometric action on an Hadamard space.
In Section~4, we show that
\begin{theo}
  \label{theo:1210142830}
  Let $G$ be a simply connected Lie group, and
  $\langle , \rangle$ be a left-invariant metric on $G$.
  If the orbit $\R_{>0}\mathrm{Aut}(\mathfrak{g}).\langle , \rangle \subset \mathfrak{m}(\mathfrak{g})$ is an 
  isolated orbit,
  then the left-invariant metric $\langle, \rangle$ is maximal.
  The converse holds if $G$ is unimodular completely solvable.
\end{theo}

In Section~5,
we construct examples of maximal metrics on Euclidean spaces which are not isotropy irreducible.
By applying Theorem~\ref{theo:1210142830}, one has
\begin{theo}
  \label{theo:1012173304}
  For $w  = (w_{2}, w_{3}, \ldots , w_{n}) \in \R^{n-1}$, define a metric $g_{w}$ on %$\R^{n} := \{(x_{1}, x_{2}, \ldots , x_{n}) \mid x_{i} \in \R\}$ by
  $\R^{n}$ with the Cartesian coordinate system $(x_{1}, x_{2}, \ldots , x_{n})$ by  
\begin{equation*}
        g_{w} := (dx_{1})^{2} + e^{-2w_{2}x_{1}} (dx_{2})^{2} + \cdots  + e^{-2w_{n}x_{1}} (dx_{n})^{2}.
      \end{equation*}
      Then $g_{w}$ is a maximal metric for all $w \in \R^{n-1}$.
      If $w_{i} \neq w_{j}$ for some $i,j \in \{2,3,\ldots , n\}$, then $g_{w}$ is not isotropy irreducible.
    \end{theo}

    The Riemannian metric $g_{w}$ is isometric to a left-invariant metric on a certain solvable Lie group.
    For more details, see Subsection~\ref{20220802165433}.
    Note that the symmetric group $\mathfrak{S}_{n-1}$ acts on $\R^{n-1}$ naturally.
      For $w, w^{\prime} \in \R^{n-1}$,
      $g_{w}$ and $g_{w^{\prime}}$ are isometric with each other if and only if
      there exists some permutation $\sigma \in \mathfrak{S}_{n-1}$ such that
      $\sigma.w = w^{\prime}$.
    Hence Theorem~\ref{theo:1012173304} gives continuously many examples of maximal metrics.
    
The other examples are  constructed by considering nilmanifolds attached with graphs.
By applying Theorem~\ref{theo:1210142830}, we show that

\begin{theo}
  \label{theo:1207153743}
  If given an edge-transitive graph $\mathcal{G}$ with $p$ vertices and $q$ edges,
  one can construct maximal metrics on $\R^{p+q}$.
  If $q \neq 0$ then the metrics are not isotropy irreducible.  
\end{theo}

A precise assertion of Theorem~\ref{theo:1207153743} will be given in Theorem~\ref{theo:0810135716}.
Note that, for graphs $\mathcal{G}$ and $\mathcal{G}^{\prime}$, %$\langle, \rangle_{\mathcal{G}}$ and $\langle, \rangle_{\mathcal{G}^{\prime}}$
the corresponding metrics
are isometric with each other if and only if $\mathcal{G}$ and $\mathcal{G}^{\prime}$ are isomorphic as graphs.
Hence, Theorem~\ref{theo:1207153743} also  gives infinitely many nontrivial examples of maximal metrics.
For more details, see Subsection~\ref{20220802165049}.
%Recall that a homogeneous maximal metric is a Ricci soliton.

%=====================SECTION STARTS HERE=====================
\section{Maximal metrics}

In this section, we give some general theory on maximal metrics.
Recall that, in Section~1,
we introduce the preordered set $(\mathfrak{M}(X)/_{\! \sim}, \prec)$.
Here, $\mathfrak{M}(X)/_{\! \sim}$ is the moduli space of Riemannian metrics on $X$ up to isometry and scaling,
and $\prec$ is the preorder with respect to the size of the isometry groups.
A maximal metric $g \in \mathfrak{M}(X)$ is the one whose equivalent class $[g] \in \mathfrak{M}(X)/_{\! \sim}$ is a maximal element
with respect to $\prec$.

\subsection{isotropy irreducible metrics and maximal metrics}
\label{20220930164546}

Firstly, we explain that isotropy irreducible metrics are maximal metrics.
Recall that, for $p \in X$ in a Riemannian manifold $(X,g)$, the action of the stabilizer
$\mathrm{Isom}(X,g)_{p}:= \{\varphi \in \mathrm{Isom}(X,g) \mid \varphi(p) = p\}$
on $T_{p}X$ by differential is called the \textit{isotropy representation} of $(X,g)$ at $p$.

\begin{defi}
  A Riemannian metric $g \in \mathfrak{M}(X)$ is called an \textit{isotropy irreducible metric}
  if the isotropy representation at each point $p \in X$ is an irreducible representation.
  A Riemannian manifold $(X,g)$ with an isotropy irreducible metric $g$ is called an \textit{isotropy irreducible space}.
\end{defi}

One can see that a complete connected isotropy irreducible space is homogeneous.
For examples, see \cite{MR1097024}.
Strongly isotropy irreducible spaces which are some special class of isotropy irreducible spaces have been classified
independently by Manturov (\cite{MR0139113, MR0139112, MR0210031}), Wolf (\cite{MR223501, MR736216}) and Kr\"{a}mer (\cite{MR376965}).
Conclusively, isotropy irreducible spaces have been classified by Wang-Ziller (\cite{MR1097024}).

\begin{prop}
  \label{prop:1102152520}
  A complete connected Riemannian manifold $(X,g)$
  is an isotropy irreducible space if and only if $(X,g)$ satisfies the following:
  \begin{equation}
    \label{eq:0721151425}
      \{h \in \mathfrak{M}(X) \mid \mathrm{Isom}(X,g) \subset \mathrm{Isom}(X,h)\} =
    \R_{>0}g.
  \end{equation}
\end{prop}

\begin{proof}
  \label{proof:1102152459}  
  We prove ``only if part''.
  Assume that $(X,g)$ is isotropy irreducible.
  Recall that an isotropy irreducible space $(X,g)$ is homogeneous.  
  Hence the assertion follows from the Schur's Lemma applied to the isotropy representation.

  We prove ``if part''.
  Firstly we show that  if $(X,g)$ satisfies the property (\ref{eq:0721151425}), then $(X,g)$ must be homogeneous.
  If $(X,g)$ is inhomogeneous,
  since $\mathrm{Isom}(X,g)$-action on the complete connected Riemannian manifold $X$ is proper,
  there exists a non-constant $\mathrm{Isom}(X,g)$-invariant smooth function $f$ on $X$.
  Then one has $e^{f} g $ is a smooth $\mathrm{Isom}(X,g)$-invariant Riemannian metric which
   is not contained in $\R_{>0}g$. This concludes that $(X,g)$ does not satisfies the property (\ref{eq:0721151425}).

  Hence we have only to consider the case $(X,g)$ is homogeneous.
  Take any $p \in X$.
  Then the set of $\mathrm{Isom}(X,g)$-invariant metrics
  on $X$ is naturally identified with
  the set of inner products which are invariant under the isotropy representation on a tangent space $T_{p}X$.
  Assume that $(X,g)$ is not isotropy irreducible.
  Then there exists a subspace $V \subsetneq T_{p}X$
  which is invariant under the isotropy representation.
  Denote by $g_1$ and $g_2$
  the restriction of $g$ to $V$ and the normal space $V^{\perp}$, respectively.
  Then $a g_1 + b g_2$ is an $\mathrm{Isom}(X,g)$-invariant metrics
  for all $a,b >0$.
  This implies that, for examples,
  $2 g_1 +  3g_2$
  is an $\mathrm{Isom}(X,g)$-invariant metrics which is not contained in $\R_{>0}g$.
\end{proof}

Note that a metric $g \in \mathfrak{M}(X)$ is maximal if and only if
\begin{equation*}
 \{h \mid \mathrm{Isom}(X,g) \subset \mathrm{Isom}(X,h)\} \subset [g], 
\end{equation*}
where $[g]$ is the equivalent class with respect to $\sim$.
Since $\R_{>0}g \subset [g]$, one has
\begin{cor}
  \label{cor:1102153754}
A complete isotropy irreducible metric is a maximal metric.
\end{cor}

\subsection{maximal metrics and self-similar solutions for metric evolution equations}
\label{20220930164756}

Denote by $\mathfrak{S}(X)$ the set of all symmetric $(0,2)$-tensors on $X$.
Let us consider a map $\mathcal{R} : \mathfrak{M}(X) \to \mathfrak{S}(X)$.
Then one can define a partial differential equation
\begin{equation*}
  \frac{\partial }{\partial t}g_{t} = \mathcal{R}(g_{t}).
\end{equation*}
For examples, the equation is called the \textit{Ricci flow} when the case $\mathcal{R} = -2\mathrm{Ric}$.

A solution $\{g_{t}\}_{t \in [0,T)}$ is called \textit{self-similar} if $[g_{t}] = [g_{0}]$ for all $t \in [0,T)$.
Also, if a metric $g \in \mathfrak{M}(X)$ admits some self-similar solution $\{g_{t}\}$ with $g = g_{0}$, then $g$ is called a \textit{soliton}.
For examples, a soliton for the Ricci flow is usually called a \textit{Ricci soliton}.
The study of self-similar solutions and solitons are important in order to understand metric evolution equations.

By the property (\ref{eq:0114133517}), one has
\begin{prop}
  \label{prop:0114134725}
  Let $g \in \mathfrak{M}(X)$ be a maximal metric.
  If a solution $\{g_{t}\}_{t \in [0,T)}$ for the equation
  $\frac{\partial}{\partial t}g_{t} = \mathcal{R}(g_{t})$
  with $g_{0} = g$ preserves the isometry group in the sense that
  $\mathrm{Isom}(X,g_{0}) \subset \mathrm{Isom}(X,g_{t})$ for all $t \in [0,T)$,
  then the solution $\{g_{t}\}_{t \in [0,T)}$ is self-similar.
  That is, the initial metric $g = g_{0}$ is soliton.
\end{prop}

For examples, one knows that the solution $\{g_{t}\} \subset \mathfrak{M}(X)$ of the Ricci flow
(\ref{eq:0712133631})
starting at $g \in \mathfrak{M}(X)$ with bounded curvature exists (\cite{MR1001277}), and preserves the isometry group (\cite{MR2260930}).
This proves Proposition~\ref{prop:0526134932}:
a homogeneous maximal metric is a Ricci soliton.

Let $(X,g)$ be a homogeneous Riemannian manifold.
Denote by $G := \mathrm{Isom}(X,g)$.
Denote by $\mathfrak{M}_{G}(X)$ and $\mathfrak{S}_{G}(X)$ the set of all $G$-invariant metrics and the set of all $G$-invariant symmetric $(0,2)$-tensors,
respectively.
Note that
$\mathfrak{S}_{G}(X)$ is a finite dimensional vector space, and $\mathfrak{M}_{G}(X)$ is an open subset of $\mathfrak{S}_{G}(X)$.
Hence, $\mathfrak{M}_{G}(X)$ and $\mathfrak{S}_{G}(X)$ have a natural differentiable structure.
Note that the group of diffeomorphism $\mathrm{Diff}(X)$ acts on $\mathfrak{S}(X)$ and $\mathfrak{M}(X)$ in the natural way.
Assume that $\mathcal{R} : \mathfrak{M}(X) \to \mathfrak{S}(X)$ is $\mathrm{Diff}(X)$-equivariant.
Since $G \subset \mathrm{Diff}(X)$,  $\mathcal{R}$ induces the map $\mathcal{R} : \mathfrak{M}_{G}(X) \to \mathfrak{S}_{G}(X)$.
The map $\mathcal{R} : \mathfrak{M}_{G}(X) \to \mathfrak{S}_{G}(X)$ can be regarded as a vector field on $\mathfrak{M}_{G}(X)$.
If the map $\mathcal{R} : \mathfrak{M}_{G}(X) \to \mathfrak{S}_{G}(X)$ is continuous, 
then there exists an integral curve $\{g_{t}\}_{t \in [0,T)} \subset \mathfrak{M}_{G}(X)$ with initial metric $g$.
Then $\{g_{t}\}_{t \in [0,T)} \subset \mathfrak{M}_{G}(X)$ is also a solution for the original differential equation $\frac{\partial}{\partial t} g_{t} = \mathcal{R}(g_{t})$.
We call the solution $\{g_{t}\}_{t \in [0,T)} \subset \mathfrak{M}_{G}(X)$ the \textit{homogeneous solution} for the equation.
Homogeneous solutions $\{g_{t}\}_{t \in [0,T)}$ preserve the isometry groups.
Hence, by Proposition~\ref{prop:0114134725}, one has

\begin{prop}
  \label{prop:1102162849}
  Let $g \in \mathfrak{M}(X)$ be a homogeneous maximal metric.
  Denote by $G := \mathrm{Isom}(X,g)$.
  Consider a $\mathrm{Diff}(X)$-equivariant map $\mathcal{R} : \mathfrak{M}(X) \to \mathfrak{S}(X)$.
  Assume that the map $\mathcal{R} : \mathfrak{M}_{G}(X) \to \mathfrak{S}_{G}(X)$ is continuous.
  Then a homogeneous solution $g_{t}$ of the partial differential equation $ \frac{\partial}{\partial t} g_{t} = \mathcal{R}(g_{t})$
  starting at $g$ is self-similar.
  In other words, $g$ is a soliton for the equation.
\end{prop}

If $\mathcal{R}$ is defined as the combination of some curvature tensors,
then $\mathcal{R}$ is $\mathrm{Diff}(X)$-equivariant, and
the restriction map $\mathcal{R} : \mathfrak{M}_{G}(X) \to \mathfrak{S}_{G}(X)$ is continuous.
Hence we can apply Proposition~\ref{prop:1102162849} for this case.

For a Riemannian metric $g$,
denote by $\mathrm{Rm}_{g}$ the Riemannian curvature tensor.
Also, $\mathrm{Ric}_{g}$ and $\mathrm{scal}_{g}$ are the Ricci curvature and the scalar curvature of $g$, respectively.
For $a,b,c \in \R$, let us define the map $\mathcal{R} : \mathfrak{M}(X) \to \mathfrak{S}(X)$ by
\begin{equation*}
  \mathcal{R}(g) = -a \mathrm{Ric}_{g} - b \mathrm{scal}_{g}\cdot g - c \mathrm{Rm}^{2}_{g} \quad (g \in \mathfrak{M}(X)).
\end{equation*}
Here, $\mathrm{Rm}^{2}_{g}$ is the $(0,2)$-symmetric tensor given by
\begin{equation*}
  \mathrm{Rm}^{2}_{g}(x,y) := - \mathrm{tr}(\mathrm{Rm}_{g}(x, \star) \circ \mathrm{Rm}_{g}(y,\star)),
\end{equation*}
where ``$\star$'' means the metric contraction with respect to $g$.
Then the equation $\frac{\partial}{\partial t}g_{t} = \mathcal{R}(g_{t})$ is called 
\begin{itemize}
\item the \textit{Ricci flow} if $a = 2$, $b = c = 0$.
\item the \textit{Yamabe flow} if $a = c = 0$, $b = 1$.
\item the \textit{Ricci-Bourguignon flow} if $a = 2$, $b \neq 0$, $c = 0$.
\item the \textit{RG-2 flow} if $a = 2$, $b = 0$, $c \neq 0$.
\end{itemize}

In these cases, $\mathcal{R}$ is a combination of the $\mathrm{Diff}(X)$-equivariant curvature tensors.
Hence one has
\begin{cor}
  \label{cor:1102164543}
  Let $g \in \mathfrak{M}(X)$ be a homogeneous maximal metric.
  Then a homogeneous solution of the metric evolution equation $\frac{\partial}{\partial t} g_{t} = \mathcal{R}(g_{t})$ with
  \begin{equation*}
    \mathcal{R}(g) = -a \mathrm{Ric}_{g} - b \mathrm{scal}_{g}\cdot g - c \mathrm{Rm}^{2}_{g} \quad (g \in \mathfrak{M}(X))
  \end{equation*}
  starting at $g$ is self-similar. In other words, a homogeneous maximal metric $g$ is a soliton for this equation.
\end{cor}

Also, one can apply Proposition~\ref{prop:1102162849} for the Bach flow.
For $g \in \mathfrak{M}(X)$, define $\mathcal{R}(g) \in \mathfrak{S}(X)$ by
\begin{equation*}
  \mathcal{R}(g)(x,y) = \frac{1}{n-3}\nabla_{\star}\nabla_{\bullet}W_{g}(x,\star,y,\bullet) + \frac{1}{n-2}\mathrm{Ric}_{g}(\star,\bullet)W_{g}(x,\star,y,\bullet),
\end{equation*}
where $W_{g}$ is the Weyl tensor with respect to $g$, and $\star$ and $\bullet$ mean the metric contraction with respect to $g$.
Then the $(0,2)$-tensor $\mathcal{R}(g)$ is called the \textit{Bach tensor} for $g$, and the equation $\frac{\partial}{\partial t}g_{t} = -\mathcal{R}(g_{t})$ is called the \textit{Bach flow}.

The Bach tensor $\mathcal{R}$ is also defined by the combination of the $\mathrm{Diff}(X)$-equivariant curvature tensors.
This yields that
\begin{cor}
  Let $g \in \mathfrak{M}(X)$ be a homogeneous maximal metric.
  Then a homogeneous solution of the Bach flow
  starting at $g$ is self-similar. In other words, a homogeneous maximal metric $g$ is a soliton for the Bach flow.
\end{cor}

\subsection{maximal metrics and maximal symmetry metrics}
\label{20220922174639}

A left-invariant metric $g$ on a Lie group $G$ is called a \textit{maximal symmetry metric} if
for all left-invariant metric $h$ on $G$ there exists $\varphi \in \mathrm{Aut}(G)$ such that $\mathrm{Isom}(G,h) \subset \mathrm{Isom}(G,\varphi.g)$,
where $\varphi.g := g(d\varphi^{-1}, d\varphi^{-1})$.
Namely, a maximal symmetry metric $g$ is the one whose isometry group is maximum among the set of left-invariant metrics up to automorphism.
The notion of maximal symmetry metric has been introduced by Jablonski-Gordon (\cite{MR3909903}).

\begin{prop}[\cite{MR2800365}]
  \label{prop:0510140912}
A left-invariant Ricci soliton metric on a simply connected unimodular completely solvable Lie group
is a maximal symmetry metric.
\end{prop}
Hence, any maximal metric on a  simply connected unimodular completely solvable Lie group has maximal symmetry.
On the other hands, when $G$ is not unimodular completely solvable, 
left-invariant maximal metrics on $G$ are not necessarily maximally symmetric.
For examples, let $\mathfrak{g} := \mathrm{span}\{v_{1}, v_{2}, v_{3}\}$ be a Lie algebra whose nonzero bracket relation is given by
\begin{equation*}
  [v_{1}, v_{2}] = v_{2}.
\end{equation*}
Note that $\mathfrak{g}$ is a non-unimodular completely solvable Lie algebra.
Let $G$ be the simply connected Lie group with Lie algebra $\mathfrak{g}$.
Then Jablonski and Gordon have shown that there are no maximal symmetry metrics on $G$ (\cite{MR3909903}).
On the other hand, the left-invariant metric with orthonormal frame $\{v_{1}, v_{2}, v_{3}\}$ is isometric to
$g_{w}$ with $w = (1,0)$ in Theorem~\ref{theo:1012173304}, and hence is a maximal metric.

\subsection{maximal metrics have the maximal isometry groups}
\label{20220922174711}

One of the interesting properties of maximal metrics $g \in \mathfrak{M}(X)$ is the maximallity of the isometry groups
in the sense that
$\mathrm{Isom}(X,\langle , \rangle) \subset \mathrm{Isom}(X,\langle , \rangle^{\prime})$ implies
$\mathrm{Isom}(X,\langle , \rangle) = \mathrm{Isom}(X,\langle , \rangle^{\prime})$ for all Riemannian metric $\langle , \rangle^{\prime}$ on $X$.
To see this, we firstly give an easy lemma, which will be used repeatedly in this article.

\begin{lemm}
  \label{lemm:1022113054}
Let $H$ be a Lie group with finitely many connected components, and assume that $H$ and a Lie group $H^{\prime}$ are isomorphic with each other.
  If $H \subset H^{\prime}$ then one has $H = H^{\prime}$.
\end{lemm}

Note that the assumption of the finiteness is essentially needed.
For examples, $2 \mathbb{Z} := \{2z \mid z \in \mathbb{Z}\} \subsetneq \mathbb{Z}$, and $2\mathbb{Z} \cong \mathbb{Z}$.

\begin{prop}
  \label{prop:1226212351}
  Let $\langle, \rangle \in \mathfrak{M}(X)$
  be a maximal metric
  whose isometry group has finitely many connected components.
  Then the isometry group $\mathrm{Isom}(X,\langle , \rangle)$
  is maximal with respect to the inclusion ``$\subset$''.
  In particular, a homogeneous maximal metric
  has the maximal isometry group.
\end{prop}

\begin{proof}
  \label{proof:0925115159}
  Take any Riemannian metric $\langle , \rangle^{\prime}$ with
  $\mathrm{Isom}(X,\langle , \rangle) \subset \mathrm{Isom}(X,\langle , \rangle^{\prime})$.
  We show that
  $\mathrm{Isom}(X,\langle , \rangle) = \mathrm{Isom}(X,\langle , \rangle^{\prime})$.
  Since $\langle, \rangle$ is maximal,
  there exists $\lambda >0$ and $\varphi \in \mathrm{Diff}(X,\langle , \rangle)$
  such that $\varphi$ is an isometry between $(X,\lambda \langle , \rangle)$
  and $(X,\langle , \rangle^{\prime})$.
  Hence one has
  \begin{equation*}
    \mathrm{Isom}(X,\langle , \rangle) = \mathrm{Isom}(X,\lambda \langle , \rangle) = \mathrm{Isom}(X,\varphi. \langle , \rangle^{\prime}) = \varphi \mathrm{Isom}(X,\langle , \rangle^{\prime}) \varphi^{-1}.
  \end{equation*}
  Hence $\mathrm{Isom}(X,\langle , \rangle)$ and $\mathrm{Isom}(X,\langle , \rangle)^{\prime}$ are isomorphic with each other.
  By Lemma~\ref{lemm:1022113054}, they coincide with each other.  
\end{proof}

% If a connected Riemannian manifold $(X,g)$ is homogeneous,
% then $\mathrm{Isom}(X,g)$ has finitely many connected components.
% Hence Proposition~\ref{prop:0526135154} in Section~1 is a corollary of Proposition~\ref{prop:1226212351}.

Here, we mention  the correspondence between maximal symmetry left-invariant metrics and maximal isometry left-invariant metrics.
Firstly, we show that
\begin{prop}
  \label{prop:0201161326}
  Let $\langle, \rangle$ be a left-invariant metric on a connected Lie group $G$.
  If $\langle, \rangle$ is a maximal symmetry metric, then $\langle, \rangle$ has the maximal isometry group.
\end{prop}

\begin{proof}
  \label{proof:0201161405}
  Assume that $\langle, \rangle$ is maximally symmetric.  
  Take any Riemannian metric $\langle, \rangle^{\prime}$ with $\mathrm{Isom}(G,\langle, \rangle) \subset \mathrm{Isom}(G,\langle, \rangle^{\prime})$.
  Since $\langle,\rangle$ is left-invariant, one has $\langle, \rangle^{\prime}$ is also left-invariant.
  Then one has $\mathrm{Isom}(G,\varphi.\langle, \rangle^{\prime}) \subset \mathrm{Isom}(G,\langle, \rangle)$ for some $\varphi \in \mathrm{Aut}(G)$.
  This yields that
  \begin{equation*}
    \mathrm{Isom}(G,\varphi.\langle, \rangle^{\prime}) \subset \mathrm{Isom}(G,\langle, \rangle) \subset \mathrm{Isom}(G,\langle, \rangle^{\prime}).
  \end{equation*}
  On the other hand, one has $\mathrm{Isom}(G,\varphi.\langle, \rangle^{\prime}) = \mathrm{Isom}(G,\langle, \rangle^{\prime})$ by Lemma~\ref{lemm:1022113054}.
  This concludes that $\langle, \rangle$ has the maximal isometry group.
\end{proof}

This concludes that the following diagram holds for arbitrary left-invariant metrics on arbitrary  Lie groups:
\begin{figure}[h]
$    \begin{array}{ccc}
    \mbox{maximal} & \Rightarrow & \mbox{maximal isometry}\\
    \Downarrow & & \Uparrow\\
    \mbox{Ricci soliton} &  & \mbox{maximal symmetry}
     \end{array}$
     \caption{for general case}
            \label{figure:0712142357}
\end{figure}

For some special cases, one can add several arrows to Figure~\ref{figure:0712142357}.
To see this,we show the following:
\begin{prop}
  \label{prop:0201162116}
  Assume that a connected Lie group $G$ admits a maximal symmetry left-invariant metric.
  Let $\langle, \rangle$ be a left-invariant metric on $G$.
  Then $\langle, \rangle$ has a maximal isometry group
  if and only if $\langle, \rangle$ is a maximal symmetry left-invariant metric.
\end{prop}

\begin{proof}
  \label{proof:0201162205}
  We have shown ``if part'' in Proposition~\ref{prop:0201161326}.
  We prove ``only if part''.
  Assume that $\langle, \rangle$ has the maximal isometry group.  
  We show that $\langle, \rangle$ is a maximal symmetry metric.  
  Take any left-invariant metric $g$ on $G$.
  We have only to show that there exists some $\varphi \in \mathrm{Aut}(G)$ such that
  $\mathrm{Isom}(G,\varphi.g) \subset \mathrm{Isom}(G,\langle, \rangle)$.
  By the assumption, there exists  a maximal symmetry  metric $\langle, \rangle^{\prime}$ on $G$.
  Then there exists $\phi, \psi \in \mathrm{Aut}(G)$ such that
  \begin{equation*}
    \mathrm{Isom}(\phi.g) \subset \mathrm{Isom}(G,\langle, \rangle^{\prime}), \quad
    \mathrm{Isom}(G,\psi.\langle, \rangle) \subset \mathrm{Isom}(G,\langle, \rangle^{\prime}).
  \end{equation*}
  This yields that $\mathrm{Isom}(G,\langle, \rangle) \subset \mathrm{Isom}(G,\psi^{-1}\langle, \rangle^{\prime})$.
  Since $\langle, \rangle$ has the maximal isometry group, 
  one has $\mathrm{Isom}(G,\psi^{-1}.\langle, \rangle^{\prime}) = \mathrm{Isom}(G,\langle, \rangle)$.
  Now define $\varphi \in \mathrm{Aut}(G)$ by $\varphi := \psi^{-1}\phi$. One has
  \begin{equation*}
    \mathrm{Isom}(G,\varphi.g) = \mathrm{Isom}(\psi^{-1}\phi.g) \subset \mathrm{Isom}(G,\psi^{-1}\langle, \rangle^{\prime}) = \mathrm{Isom}(G,\langle, \rangle),
  \end{equation*}
  which completes the proof.
\end{proof}

For examples, Jablonski and Gordon has shown that any unimodular completely solvable Lie groups always admit maximally symmetric left-invariant metrics.
Hence, for the case of unimodular completely solvable Lie groups, one has

\begin{figure}[htbp]
$    \begin{array}{ccc}
    \mbox{maximal} & \Rightarrow & \mbox{maximal isometry}\\
    \Downarrow & & \Updownarrow\\
    \mbox{Ricci soliton} & \Rightarrow & \mbox{maximal symmetry}
     \end{array}$
     \caption{for unimodular completely solvable case}
  \label{figure:0712142537}     
\end{figure}

%=====================SECTION STARTS HERE=====================
\section{Some general theory on isolated orbits}

In this section, 
we construct some general theory on isolated orbits of proper actions.
Our goal is to show Proposition~\ref{prop:0909223530}, which plays an important role to construct examples of maximal metrics. 

Firstly, we introduce some notions of orbits of set-theoretical group actions.
Let $G$ be a group, and assume that $G$ acts on a set $X$.
Denote by $G \backslash X$ the orbit space of the $G$-action.
For orbits $G.p, G.q \in G \backslash X$, denote by $G.p \sim G.q$ (\textit{resp.} $G_{p} < G_{q}$) if there exists $q^{\prime} \in G.q$ such that $G_{p} = G_{q}$
(\textit{resp.} $G.p \subset G.q$).
Then $\sim$ is an equivalent relation on $G \backslash X$,
and $<$ is a preorder on $G \backslash X$.
An orbit $G.p$ is said to be  \textit{of maximal orbit type} if $G.p$ is maximal with respect to $<$, namely,
if $G.p < G.q$ then $G.p = G.q$.
Denote by $[G.p]$ the equivalent class of $G.p$ with respect to $\sim$.
An orbit $G.p$ is called a \textit{solitary orbit} if $[G.p] = \{G.p\}$.
By the definition, if $G.p$ has a maximal orbit type then $G.p$ is a solitary orbit.
The preorder $<$ on $G \backslash X$ induces
a preorder on the double coset space
\begin{equation*}
  G\backslash X /_{\! \sim} := \{[G.p] \mid G.p \in G\backslash X\},
\end{equation*}
which is also denoted by $<$.
Namely, we denote by
$[G.p] < [G.q]$ if $G.p < G.q^{\prime}$ for some $q^{\prime} \in G.q$.

Now we consider group actions with topology.
Let $G$ be a topological group, and assume that $G$ acts on a topological space $X$ continuously.
We consider the natural quotient topology on $G \backslash X$.
An orbit $G.p$ is called an \textit{isolated orbit} if $G.p \in G \backslash X$ is an isolated point of the subset $[G.p] \subset G \backslash X$,
that is, there exists some open subset $U \subset G \backslash X$ such that $U \cap [G.p] = \{G.p\}$.
In other words, an isolated orbit is a locally solitary orbit.
Now one has
\begin{equation}
  \label{eq:0910002409}
   \begin{array}{ccccc}
   \mbox{maximal orbit type}&\Rightarrow &\mbox{solitary}& \Rightarrow &\mbox{isolated}.
 \end{array}
\end{equation}
In general, these notions are not equivalent.
However, these notions are equivalent for some special case (\textit{cf.} Proposition~\ref{prop:0909223530}).

We review some fundamental facts on proper actions.
A continuous action of a topological group $G$ on a topological space $X$ is called \textit{proper} if the map
\begin{equation*}
  G \times X \to X \times X, \ (g,x) \mapsto (g.x,x)
\end{equation*}
is a proper map (\textit{i.e.} inverse images of compact subsets are also compact).
By the definition, one can see  that isotropy subgroups of a proper action on
a manifold are compact,
and hence have finitely many connected components.
By using Lemma~\ref{lemm:1022113054}, one can easily see that
\begin{lemm}
  \label{lemm:0910010435}
  Let $X$ be a manifold, and assume that a Lie group $G$ acts on $X$ properly.
  Then the preordered set $(G\backslash X /_{\! \sim},<)$ is a partially ordered set.
  Namely, $[G.p] < [G.q]$ and $[G.q] < [G.p]$ imply $[G.p] = [G.q]$.
\end{lemm}

Now let $(X,\langle, \rangle)$ be a complete connected Riemannian manifold, and $G \subset \mathrm{Isom}(X,\langle, \rangle)$ be a closed subgroup.
Then it is known that $G$ acts on $X$ isometrically and properly.
Also, by the existence of slices for proper actions on manifolds (\textit{e.g.} \cite{MR3362465}),
one has
\begin{prop}
  \label{prop:0910000918}
  Let $(X,\langle, \rangle)$ be a complete connected Riemannian manifold, and $G \subset \mathrm{Isom}(X,\langle, \rangle)$ be a closed subgroup.
  For each orbit $G.p \in G \backslash X$, there exists an open neighborhood $V \subset G \backslash X$ of $G.p$ such that
  $G.x < G.p$ for all $G.x \in V$.
\end{prop}

For $G.p \in G \backslash X$, denote by $\mathcal{U}(G.p)$ the set of all upper bounds of $G.p$, namely
\begin{equation*}
  \mathcal{U}(G.p) := \{G.q \in G \backslash X \mid G.p < G.q\}.
\end{equation*}
Then one has the following:

\begin{lemm}
  \label{lemm:0915152534}
  Let $(X,\langle, \rangle)$ be a Riemannian manifold whose exponential map $\mathrm{exp}_{p} : T_{p}X \to X$ is a diffeomorphism for all $p \in X$.
  Let $G \subset \mathrm{Isom}(X,\langle,\rangle)$ be  a closed subgroup.
  Then $\mathcal{U}(G.p)$ is a connected subset of $G \backslash X$ for all $G.p \in G\backslash X$.
\end{lemm}

\begin{proof}
  \label{proof:0915233019}
  We prove that $\mathcal{U}(G.p)$ is path-connected.
  Take any $G.q \in G \backslash X$ with $G.p < G.q$.
  We construct a path from $G.p$ to $G.q$.
  That is,
  we show that
  there exists a continuous curve $\gamma : [0,T] \to G \backslash X$ such that
  \begin{equation}
    \label{eq:0915233702}
    \gamma(0) = G.p, \quad \gamma(T) = G.q, \quad  G.p = \gamma(0) < \gamma(t) \quad (t \in [0,T]).
  \end{equation}

  By the definition of $<$, there exists $q^{\prime} \in G.q$ such that $G_{p} \subset G_{q^{\prime}}$.
  Take any geodesic $c : [0,T] \to X$ with $c(0) = p$, $c(T) = q^{\prime}$.
  Since the exponential map $\mathrm{exp}_{p} : T_{p}X \to X$ is a diffeomorphism,  
  such a geodesic $c$ is unique.
  Note that $G_{c(0)} = G_{p} \subset G_{q^{\prime}} = G_{c(T)}$.
  That is, the action of $G_{p}$ on $X$ fixes the starting point $c(0)$ and the end point $c(T)$ of $c$.
  We also note that the $G_{p}$-action sends each geodesic to geodesic.
  Hence, by the uniqueness of the geodesic $c$, the $G_{p}$-action must fix the geodesic $c$ pointwisely.
  Therefore, one has  $G_{p} \subset  G_{c(t)}$, namely $G.p < G.c(t)$ for all $t \in [0,T]$.
   Let $\pi : X \to G \backslash X$ be the projection, and define $\gamma := \pi \circ c$.
   Then one can see that $\gamma$ satisfies (\ref{eq:0915233702}).
\end{proof}

By the above preparation, we prove the main subject of this section as follows:

\begin{prop}
  \label{prop:0909223530}
  Let $(X,\langle, \rangle)$ be a Riemannian manifold whose exponential map $\mathrm{exp}_{p} : T_{p}X \to X$ is a diffeomorphism for all $p \in X$.
  Let $G \subset \mathrm{Isom}(X,\langle,\rangle)$ be a closed subgroup.
  If an orbit $G.p$ is isolated, then $G.p$ has a maximal orbit type.
  In particular, three notions in (\ref{eq:0910002409}) are equivalent with each other.
\end{prop}

\begin{proof}
  \label{proof:0916001815}
  Assume that an orbit $G.p$ is isolated.
  We have to show $G.p$ has maximal orbit type, namely,  $\mathcal{U}(G.p) = \{G.p\}$.
  By Lemma~\ref{lemm:0915152534}, one knows $\mathcal{U}(G.p)$ is connected.
  Therefore, in order to show $\mathcal{U}(G.p) = \{G.p\}$, we have only to show that $\{G.p\}$ is a clopen subset of $\mathcal{U}(G.p)$.

  We show that $\{G.p\}$ is a closed subset of $\mathcal{U}(G.p)$.
  The orbit space $G \backslash X$ is Hausdorff since $G$-action is proper.
  Therefore, $\{G.p\} \subset G \backslash X$ is closed, and hence is closed in $\mathcal{U}(G.p)$.

  We now verify that $\{G.p\}$ is an open subset of $\mathcal{U}(G.p)$.
  To show this, we construct an open subset $W$ of $G \backslash X$ such that $W \cap \mathcal{U}(G.p) = \{G.p\}$.

  We firstly construct $W$.
  Since $G.p$ is isolated, there exists an open subset $U$ such that $U \cap [G.p] = \{G.p\}$.
  Also, by Proposition~\ref{prop:0910000918}, there exists an open neighborhood $V$ of $G.p$ such that
  $G.x < G.p$ for all $G.x \in V$.
  Denote by $W := U \cap V$.

  We show that $W \cap \mathcal{U}(G.p) = \{G.p\}$.
  $W \cap \mathcal{U}(G.p) \supset \{G.p\}$ is obvious.
  We have only to verify that $W \cap \mathcal{U}(G.p) \subset U \cap [G.p]$
  since $U \cap [G.p] = \{G.p\}$.  
  Take any $G.q \in W \cap \mathcal{U}(G.p)$.
  Since $W \subset U$, one has $G.q \in U$.
  We lastly show that $G.q \in [G.p]$.
  Since $G.q \in V$, one has $G.q < G.p$.
  This yields that $[G.q] < [G.q]$.
  On the other hand, since $G.q \in \mathcal{U}(G.p)$, one has $G.p < G.q$. Hence one has $[G.p] < [G.q]$.
  Therefore, by Lemma~\ref{lemm:0910010435}, we have $[G.p] = [G.q]$, and hence $G.q \in [G.p]$.  
\end{proof}

% =====================SECTION ENDS HERE=======================

%=====================SECTION STARTS HERE=====================
\section{Some sufficient conditions for maximal left-invariant metrics}

In this section,
we prove Theorem~\ref{theo:1210142830} which  gives a simple sufficient condition
for left-invariant metrics on a simply connected Lie group to be maximal.
%which is given in Theorem~\ref{theo:0910055521}.

\subsection{An introduction to $\R_{>0}\mathrm{Aut}(\mathfrak{g})$-actions}

Firstly, we give a review on the study of moduli spaces of left-invariant metrics.
Let $G$ be a simply connected Lie group with the Lie algebra $\mathfrak{g}$.
Note that a left-invariant metric on $G$ is naturally identified with the inner product on $\mathfrak{g}$.
Let $\mathfrak{m}(\mathfrak{g})$ be the set of all inner products on $\mathfrak{g}$.
Then the general linear group $\mathrm{GL}(\mathfrak{g})$ acts transitively on $\mathfrak{m}(\mathfrak{g})$ by
\begin{equation*}
  g.\langle \cdot, \cdot \rangle :=
  \langle g^{-1} \cdot, g^{-1} \cdot \rangle \quad (g \in \mathrm{GL}(\mathfrak{g}), \ \langle , \rangle \in \mathfrak{M}(\mathfrak{g})).
\end{equation*}
Denote by $\R_{>0}\mathrm{Aut}(\mathfrak{g}) := \{c \varphi \in \mathrm{GL}(\mathfrak{g}) \mid c >0, \ \varphi \in \mathrm{Aut}(\mathfrak{g})\}$.
Note that the subgroup $\R_{>0}\mathrm{Aut}(\mathfrak{g})$ is closed in $\mathrm{GL}(\mathfrak{g})$, and  acts on $\mathfrak{m}(\mathfrak{g})$.

Since $G$ is simply connected, the differential map $\mathrm{Aut}(G) \to \mathrm{Aut}(\mathfrak{g})$ is bijective.
This yields the following:
\begin{lemm}
  \label{lemm:0910132712}
  Let $G$ be a simply connected Lie group with the Lie algebra $\mathfrak{g}$.
  If two inner products $\langle , \rangle, \langle , \rangle^{\prime} \in \mathfrak{m}(\mathfrak{g})$ belong to the same $\R_{>0}\mathrm{Aut}(\mathfrak{g})$-orbit then
two left-invariant metrics $\langle , \rangle, \langle , \rangle^{\prime} \in \mathfrak{M}(G)$ are isometric up to scaling.
\end{lemm}
Note that the converse is not true in general, but is true if $\mathfrak{g}$ is unimodular completely solvable case for instance.
The orbit space $\R_{>0}\mathrm{Aut}(\mathfrak{g}) \backslash \mathfrak{M}(\mathfrak{g})$ helps us to classify
left-invariant metrics on $G$, and some geometers have been studying the $\R_{>0}\mathrm{Aut}(\mathfrak{g})$-action (\textit{e.g.} \cite{MR2783396, MR3431025}).

Denote by $O(\langle, \rangle) \subset \mathrm{GL}(\mathfrak{g})$ the orthogonal group with respect to $\langle, \rangle \in \mathfrak{m}(\mathfrak{g})$.
Now we investigate the stabilizer $\R_{>0}\mathrm{Aut}(\mathfrak{g}) \cap O(\langle , \rangle)$.
One can easily see that $\mathrm{Aut}(\mathfrak{g}) \cap O(\langle,\rangle) \subset \R_{>0}\mathrm{Aut}(\mathfrak{g}) \cap O(\langle , \rangle)$.
In fact, the equality holds:

\begin{lemm}
  \label{lemm:0909002910}
    Let $\mathfrak{g}$ be a Lie algebra, and $\langle , \rangle$ be an inner product on $\mathfrak{g}$.
  Then one has $\R_{>0}\mathrm{Aut}(\mathfrak{g}) \cap O(\langle , \rangle) = \mathrm{Aut}(\mathfrak{g}) \cap O(\langle , \rangle)$.
  In other words,
  the isotropy subgroup of the $\R_{>0}\mathrm{Aut}(\mathfrak{g})$-action at $\langle , \rangle \in \mathfrak{m}(\mathfrak{g})$
  coincides with $\mathrm{Aut}(\mathfrak{g}) \cap O(\langle , \rangle)$.
\end{lemm}

\begin{proof}
  \label{proof:0909003051}
  We have to show
  $\R_{>0}\mathrm{Aut}(\mathfrak{g}) \cap O(\langle , \rangle) \subset \mathrm{Aut}(\mathfrak{g}) \cap O(\langle , \rangle)$ only.
  To prove this, let us prepare some notions.

  We firstly define the action of $\mathrm{GL}(V)$ on the set of Lie brackets on $V$ by
  \begin{equation*}
    g.[,] := g[g^{-1},g^{-1}] \quad (g \in \mathrm{GL}(V), \  [,] : \mbox{Lie bracket}).
  \end{equation*}

  Secondly, for a metric Lie algebra $(V,[,],\langle , \rangle)$, denote by $\|[,]\|_{\langle , \rangle}$ the norm of the Lie bracket $[,]$, namely,
  \begin{equation}
    \label{eq:0912021357}
    \| [,] \|_{\langle , \rangle} := (\sum_{i,j} \langle [v_{i},v_{j}], [v_{i},v_{j}] \rangle)^{1/2},
  \end{equation}
  where $\{v_{1}, \ldots , v_{n}\} \subset V$ is an orthonormal basis with respect to $\langle , \rangle$.  

  We are now in the position to show the assertion.
  For a Lie algebra $\mathfrak{g} = (V,[,])$ with an inner product $\langle, \rangle$, one can see that
  \begin{align}
    \label{align:0909005025}
\begin{split}
    \mathrm{Aut}(\mathfrak{g}) &= \{g \in \mathrm{GL}(V) \mid g.[,] = [,]  \}, \\
    \R_{>0}\mathrm{Aut}(\mathfrak{g}) &= \{g \in \mathrm{GL}(V) \mid \| [,] \|_{\langle , \rangle} \cdot g.[,] = \| g.[,] \|_{\langle , \rangle} \cdot [,]  \}.    
\end{split}
  \end{align}  
 On the other hand, one can see that $\| g.[,] \|_{\langle , \rangle} = \| [,] \|_{\langle , \rangle}$ for all $g \in \mathrm{O}(\langle , \rangle)$.
 These conclude that $\R_{>0}\mathrm{Aut}(\mathfrak{g}) \cap O(\langle , \rangle) \subset \mathrm{Aut}(\mathfrak{g}) \cap O(\langle , \rangle)$.
\end{proof}

Denote by $\mathrm{sym}(\mathfrak{g})$ the vector space  of all symmetric bilinear forms on $\mathfrak{g}$.
Then $\mathfrak{m}(\mathfrak{g})$ is an open subset of $\mathrm{sym}(\mathfrak{g})$, and hence
each tangent space $T_{\langle , \rangle}\mathfrak{m}(\mathfrak{g})$ at $\langle , \rangle \in \mathfrak{m}(\mathfrak{g})$ is naturally identified with
$\mathrm{sym}(\mathfrak{g})$.
A natural inner product $(,)_{\langle , \rangle}$ on $T_{\langle , \rangle}\mathfrak{m}(\mathfrak{g}) \cong \mathrm{sym}(\mathfrak{g})$ is defined as follows:
\begin{equation*}
  (\theta, \eta)_{\langle , \rangle} := \sum_{i}\theta(v_{i},v_{j})\eta(v_{i},v_{j})
  \quad (\theta, \eta \in \mathrm{sym}(\mathfrak{g})),
\end{equation*}
where $\{v_{1}, \ldots , v_{n}\}$ is an orthonormal basis with respect to $\langle , \rangle$.
Then $(\mathfrak{m}(\mathfrak{g}),(,))$ is a $\mathrm{GL}(\mathfrak{g})$-homogeneous Riemannian manifold.
Note that $\mathfrak{m}(\mathfrak{g}) \cong \R_{>0} \times \mathrm{SL}(n)/\mathrm{SO}(n)$ by the de Rham decomposition,
and hence $\mathfrak{m}(\mathfrak{g})$ is an Hadamard manifold.

\begin{lemm}
  \label{lemm:0910060429}
$\R_{>0}\mathrm{Aut}(\mathfrak{g})$ is a closed subgroup of 
the isometry group of $(\mathfrak{m}(\mathfrak{g}),(,))$.
In particular, the $\R_{>0}\mathrm{Aut}(\mathfrak{g})$-action is a proper isometric action on the Hadamard manifold $\mathfrak{m}(\mathfrak{g})$.
\end{lemm}

\begin{proof}
  \label{proof:0910060528}
  The de Rham decomposition of $\mathfrak{m}(\mathfrak{g}) \cong \R_{>0} \times \mathrm{SL}(n)/\mathrm{SO}(n)$ implies that
  the identity component of $\mathrm{Isom}(\mathfrak{m}(\mathfrak{g}),(,))$
  is $\mathrm{GL}_{+}(\mathfrak{g}) \cong \R_{>0} \times \mathrm{SL}(\mathfrak{g})$.
  This yields that  $\mathrm{GL}(\mathfrak{g})$ is a closed subgroup of $\mathrm{Isom}(\mathfrak{m}(\mathfrak{g}),(,))$.
  On the other hand, 
  by the expression (\ref{align:0909005025}), $\R_{>0}\mathrm{Aut}(\mathfrak{g})$ is a closed subgroup of $\mathrm{GL}(\mathfrak{g})$.
  Hence $\R_{>0}\mathrm{Aut}(\mathfrak{g})$ is closed in $\mathrm{Isom}(\mathfrak{m}(\mathfrak{g}),(,))$.
\end{proof}

\subsection{some sufficient conditions for maximal metrics}  

In this subsection, we show the first assertion of Theorem~\ref{theo:1210142830} by proving Theorem~\ref{theo:0910055521}, and
show the second assertion by proving Theorem~\ref{theo:1130135127}.
Firstly, we give a following sufficient condition 
for left-invariant metrics on a simply connected Lie group to be maximal.

\begin{prop}
  \label{prop:0905222855}
  Let $G$ be a simply connected Lie group, and
  $\langle , \rangle$ be a left-invariant metric on $G$.
  Denote by $\mathfrak{g}$ the Lie algebra of $G$.
  If the orbit $\R_{>0}\mathrm{Aut}(\mathfrak{g}).\langle , \rangle$ has the maximal orbit type (see Section~3),
  then the equivalent class $[\langle, \rangle] \in \mathfrak{M}(G)/_{\! \sim}$
  is a maximal element. In other words, $\langle, \rangle$ is a maximal metric.
\end{prop}

\begin{proof}
  \label{proof:0906210631}
  Assume that $\R_{>0}\mathrm{Aut}(\mathfrak{g}).\langle , \rangle$ has the maximal orbit type.
  Take any $\mathrm{Isom}(G,\langle,\rangle)$-invariant metric $\langle, \rangle^{\prime} \in \mathfrak{M}(G)$.  
  Since $\mathrm{Isom}(G,\langle, \rangle) \subset \mathrm{Isom}(G,\langle,\rangle^{\prime})$,  
  the metric $\langle, \rangle^{\prime}$ is also a left-invariant metric on $G$.
  To show that $\langle, \rangle$ and $\langle, \rangle^{\prime}$ are isometric up to scaling,
  we have only to show that $\R_{>0}\mathrm{Aut}(\mathfrak{g}).\langle, \rangle = \R_{>0}\mathrm{Aut}(\mathfrak{g}).\langle, \rangle^{\prime}$.
  One has 
\begin{equation*}
  \mathrm{Isom}(G,\langle, \rangle) \cap \mathrm{Aut}(G) \subset \mathrm{Isom}(G,\langle,\rangle^{\prime}) \cap \mathrm{Aut}(G).
\end{equation*}
  Also, since $G$ is simply connected, one has
  $\mathrm{Isom}(G,g) \cap \mathrm{Aut}(G) \cong O(g) \cap \mathrm{Aut}(\mathfrak{g})$ for any left-invariant metric $g$ via the differential at the unit element $e \in G$.
  Here, $O(g) \subset \mathrm{GL}(\mathfrak{g})$ is the orthogonal group with respect to the inner product $g$.
  This yields that
  \begin{equation*}
   O(\langle, \rangle) \cap \mathrm{Aut}(\mathfrak{g}) \subset O(\langle, \rangle^{\prime}) \cap \mathrm{Aut}(\mathfrak{g}). 
  \end{equation*}
  Note that $O(g) \cap \mathrm{Aut}(\mathfrak{g})$ is the stabilizer of the $\R_{>0}\mathrm{Aut}(\mathfrak{g})$-action at $g \in \mathfrak{m}(\mathfrak{g})$.
  Hence one has $\R_{>0}\mathrm{Aut}(\mathfrak{g}).\langle, \rangle < \R_{>0}\mathrm{Aut}(\mathfrak{g}).\langle, \rangle^{\prime}$.
  By the maximallity, one has $\R_{>0}\mathrm{Aut}(\mathfrak{g}).\langle, \rangle = \R_{>0}\mathrm{Aut}(\mathfrak{g}).\langle, \rangle^{\prime}$.
\end{proof}

If the orbit space $\R_{>0}\mathrm{Aut}(\mathfrak{g}) \backslash \mathfrak{m}(\mathfrak{g})$
is small, then one can reasonably check the maximality condition.
For examples, Proposition~\ref{prop:0905222855} shows that if the $\R_{>0}\mathrm{Aut}(\mathfrak{g})$-action is transitive
(\textit{i.e.} the orbit space is a point), then any left-invariant metric on $G$ is maximal.
Note that Lie algebras $\mathfrak{g}$ whose $\R_{>0}\mathrm{Aut}(\mathfrak{g})$-actions are transitive have been classified
by Lauret as follows:
\begin{theo}[\cite{MR1958155}]
  \label{theo:0910151743}
  Let $\mathfrak{g}$ be an $n$-dimensional Lie algebra.
  If the $\R_{>0}\mathrm{Aut}(\mathfrak{g})$-action is transitive, then $\mathfrak{g}$ is isomorphic to one of the following three Lie algebras:
  \begin{enumerate}
    [$(1)$]
  \item abelian Lie algebra $\R^{n}$,
  \item almost abelian Lie algebra $\R \ltimes_{\varphi} \R^{n-1}$ with $\varphi(t) := t \cdot \mathrm{id}_{\R^{n-1}}$.
    In other words, the Borel subalgebra of $\mathfrak{so}(n-1,1)$.
  \item the product algebra $\R^{n-3} \oplus \mathfrak{h}_{3}$ of the abelian Lie algebra $\R^{n-3}$ and the Heisenberg Lie algebra $\mathfrak{h}_{3}$.
  \end{enumerate}
\end{theo}

Note that $(1)$ and $(2)$ in Theorem~\ref{theo:0910151743} are isotropy irreducible.
On the other hand, the case $(3)$ is not isotropy irreducible.
This gives first nontrivial examples of maximal metrics:
\begin{cor}
  \label{cor:0910153509}
  Let $G := \R^{n-3} \times H_{3}$ be a product Lie group of the abelian Lie group $\R^{n-3}$ and the Heisenberg Lie group $H_{3}$.
  Then any left-invariant metric on $G$ is a maximal metric which is not isotropy irreducible.
\end{cor}

\begin{rem}
  \label{rem:0714141634}
  Note that a simply connected nilpotent Lie group $G$ with left-invariant metric $\langle, \rangle$ is isotropy irreducible if and only if $G$ is abelian.
  By the result of Wilson (\cite{MR661539}), the isotropy representation of $(G,\langle, \rangle)$ at the unit element $e$ is
  the action of $\mathrm{Aut}(\mathfrak{g}) \cap O(\langle, \rangle)$ on $\mathfrak{g}$, where $\mathfrak{g} = T_{e}G$ is the Lie algebra of $G$.
  If $\mathfrak{g}$ is non-abelian and nilpotent, there is a nontrivial center $\mathfrak{z}(\mathfrak{g}) \subset \mathfrak{g}$, which is an invariant subspace of the isotropy
  representation.
\end{rem}

Recall that $\R_{>0}\mathrm{Aut}(\mathfrak{g})$-actions are proper isometric actions on the Hadamard manifolds $\mathfrak{m}(\mathfrak{g})$.
By applying Proposition~\ref{prop:0909223530} and Proposition~\ref{prop:0905222855}, one has
the first assertion of Theorem~\ref{theo:1210142830}:

\begin{theo}
  \label{theo:0910055521}
  Let $G$ be a simply connected Lie group, and
  $\langle , \rangle$ be a left-invariant metric on $G$.
  Denote by $\mathfrak{g}$ the Lie algebra of $G$.
  If the orbit $\R_{>0}\mathrm{Aut}(\mathfrak{g}).\langle , \rangle$ is an isolated orbit,
  then $\langle, \rangle$ is a maximal metric.
\end{theo}

It has been known that singular orbits of cohomogeneity one actions are isolated orbits. 
It has been shown that the cohomogeneity of $\R_{>0}\mathrm{Aut}(\mathfrak{g})$-actions for $3$-dimensional solvable $\mathfrak{g}$
are at most one, and some of them are cohomogeneity one actions with singular orbits.
For examples,
the Lie algebra of the motion group of $\R^{2}$ which is given by $\mathfrak{g} := (\mathrm{span}\{v_{1}, v_{2}, v_{3}\},[,])$,
\begin{equation*}
  [v_{i}, v_{j}] = 0 \ (i,j \in \{1,2\}), \quad [v_{1}, v_{3}] = -v_{2}, \quad [v_{2}, v_{3}] = v_{1}
\end{equation*}
gives an example of $\R_{>0}\mathrm{Aut}(\mathfrak{g})$-action with a singular orbit.
For more details, see \cite{MR3663797}.

Note that an $\R_{>0}\mathrm{Aut}(\mathfrak{g})$-orbit being an isolated orbit is not a necessary condition for left-invariant to be a maximal metric (see Remark~\ref{rem:0916042309}).
However, if $\mathfrak{g}$ is a unimodular completely solvable Lie algebra,
then the isolation of the orbit is also a necessary condition:

\begin{theo}
  \label{theo:1130135127}
  Let $G$ be a simply connected unimodular completely solvable Lie group, and
  $\langle , \rangle$ be a left-invariant metric on $G$.
  Denote by $\mathfrak{g}$ the Lie algebra of $G$.
  Then $\langle, \rangle$ is a maximal metric if and only if
  the orbit $\R_{>0}\mathrm{Aut}(\mathfrak{g}).\langle , \rangle$ is an isolated orbit.
\end{theo}

\begin{proof}
  \label{proof:1130135929}
  Assume that $\langle, \rangle$ is maximal.
  By the diagram (\ref{eq:0910002409}), we have only to show that $\R_{>0}\mathrm{Aut}(\mathfrak{g}).\langle, \rangle$ has maximal orbit type.
  Namely
  $O(\langle, \rangle) \cap\mathrm{Aut}(\mathfrak{g}) \subset  O(\langle, \rangle^{\prime}) \cap\mathrm{Aut}(\mathfrak{g})$
  implies $\R_{>0}\mathrm{Aut}(\mathfrak{g}).\langle, \rangle = \R_{>0}\mathrm{Aut}(\mathfrak{g}).\langle, \rangle^{\prime}$.

  Take any left-invariant metric $\langle, \rangle^{\prime}$ with
  $O(\langle, \rangle) \cap\mathrm{Aut}(\mathfrak{g}) \subset  O(\langle, \rangle^{\prime}) \cap\mathrm{Aut}(\mathfrak{g})$.
  When $G$ is simply connected unimodular completely solvable, it has been known that
  \begin{equation*}
   \mathrm{Isom}(G,g) \cong (O(g) \cap \mathrm{Aut}(\mathfrak{g})) \ltimes L_{G} 
  \end{equation*}
  for any left-invariant metric $g$ on $G$, where $L_{G}$ is the group of left-translations of $G$ (\cite{MR936815}).
  This yields that
  $O(\langle, \rangle) \cap\mathrm{Aut}(\mathfrak{g}) \subset  O(\langle, \rangle^{\prime}) \cap\mathrm{Aut}(\mathfrak{g})$ implies
  $\mathrm{Isom}(G,\langle, \rangle) \subset \mathrm{Isom}(G,\langle, \rangle^{\prime})$.
  Since $\langle, \rangle$ is maximal,
  there exists $\lambda > 0$ such that $(G,\lambda \langle, \rangle)$ and $(G,\langle, \rangle^{\prime})$ are isometric.
  It has been known that if two completely solvable Riemannian Lie groups $(H,g)$ and $(H^{\prime}, g^{\prime})$
  are isometric then there exists a group isomorphism $\varphi : H \to H^{\prime}$ such that $\varphi$ is also a Riemannian isometry (\cite{MR0362145}).
  This yields that there exists an automorphism $\psi : G \to G$ such that
  $\psi$ is an isometry between
  $(G,\lambda \langle, \rangle)$ and $(G,\langle, \rangle^{\prime})$.
  This concludes that $\R_{>0}\mathrm{Aut}(\mathfrak{g}).\langle, \rangle = \R_{>0}\mathrm{Aut}(\mathfrak{g}).\langle, \rangle^{\prime}$.
\end{proof}

\begin{rem}
  \label{rem:0916042309}
  The converse of Theorem~\ref{theo:0910055521} does not hold in general. For examples,
  let us consider the Lie algebra $\mathfrak{s} := \mathrm{span}\{x_{1}, \ldots , x_{n}, y_{1}, \ldots y_{n}, z, v\}$ by
  \begin{equation*}
    [x_{i}, y_{i}] = z, \ [x_{i}, v] = (1/2)x_{i}, \ [y_{i}, v] = (1/2)y_{i}, \ [z,v] = z.
  \end{equation*}  
  Denote by $\langle, \rangle$ an inner product on $\mathfrak{s}$ with an orthonormal basis $\{x_{1}, \ldots, y_{n}, z, v\}$.
  Then the simply connected Riemannian Lie group with respect to the metric Lie algebra $(\mathfrak{s}, \langle , \rangle)$
  is isometric to complex hyperbolic space $\mathbb{C}H^{n}$ up to scaling.
  Since $\mathbb{C}H^{n}$ is isotropy irreducible, then the left-invariant metric $\langle, \rangle$ is maximal.
  On the other hand, one can see that the orbit $\R_{>0}\mathrm{Aut}(\mathfrak{s}).\langle, \rangle$ is not isolated.  
\end{rem}

% =====================SECTION ENDS HERE=======================

%=====================SECTION STARTS HERE=====================
\section{examples}

% In this section, we give examples of left-invariant maximal metrics by using Theorem~\ref{theo:1130135127}.
In this section,
by applying Theorem~\ref{theo:1210142830},
we prove Theorem~\ref{theo:1012173304} and Theorem~\ref{theo:1207153743} that give examples of left-invariant maximal metrics on some solvable Lie groups.

\subsection{general settings}

According to Theorem~\ref{theo:1210142830},
in order to construct examples of left-invariant maximal metrics,
we have only to find a metric Lie algebra $(\mathfrak{g},\langle, \rangle)$
so that the orbit $\R_{>0}\mathrm{Aut}(\mathfrak{g}).\langle, \rangle$ is an isolated orbit in $\mathfrak{M}(\mathfrak{g})$.
The following characterization of isolated orbits was obtained by the author:

\begin{prop}[\cite{MR3771997}]
  \label{prop:0728212900}
  Let $X$ be a complete connected Riemannian manifold, and $G \subset \mathrm{Isom}(X)$ be a closed subgroup.
  Then an orbit $G.p$ is an isolated orbit if and only if the slice representation at $p$ has no nonzero fixed normal vectors.
\end{prop}

Here, the slice representation of an orbit $G.p \subset X$ at a point $p \in X$ is a linear representation of the stabilizer $G_{p}$ on the normal space $(T_{p}G.p)^{\perp}$ by differential.

Therefore, we construct metric Lie algebras $(\mathfrak{g},\langle , \rangle)$
such that the slice representation of the $\R_{>0}\mathrm{Aut}(\mathfrak{g})$-action at $\langle , \rangle$
has no nonzero fixed normal vectors.
Recall that the tangent space at $\langle , \rangle \in \mathfrak{M}(\mathfrak{g})$ is naturally identified with
the space of symmetric bilinear forms $\mathrm{sym}(\mathfrak{g})$.
By fixing an orthonormal basis $\{v_{1}, \ldots, v_{n}\} \subset \mathfrak{g}$,
one can see that the normal space $(T_{\langle , \rangle}\R_{>0}\mathrm{Aut}(\mathfrak{g}).\langle , \rangle)^{\perp}$ is given by
\begin{equation}
  \label{eq:0728212928}
    \{\theta \in \mathrm{sym}(\mathfrak{g}) \mid  \forall A \in \R \oplus \mathrm{Der}(\mathfrak{g}),\  \sum_{i}\theta(v_{i},A v_{i}) = 0\},
  \end{equation}
  where $\R \oplus \mathrm{Der}(\mathfrak{g}) := \{c \cdot \mathrm{id}_{\mathfrak{g}} + D \mid c \in \R, \ D \in \mathrm{Der}(\mathfrak{g})\}$ is the Lie algebra of $\R_{>0}\mathrm{Aut}(\mathfrak{g})$.
The slice representation at $\langle , \rangle$ is given by $g.\theta := \theta(g^{-1}, g^{-1})$, where $g \in \mathrm{Aut}(\mathfrak{g}) \cap \mathrm{O}(\langle , \rangle)$, and $\theta \in (T_{\langle , \rangle}\R_{>0}\mathrm{Aut}(\mathfrak{g}).\langle , \rangle)^{\perp}$. %(see Lemma~\ref{lemm:0821143712}).
In conclusion, we have the following linear algebraic sufficient condition for left-invariant metrics to be maximal:

\begin{prop}
  \label{prop:0803132739}
  Let $\mathfrak{g}$ be an $n$-dimensional Lie algebra, 
  $\mathcal{B} = \{v_{1}, v_{2}, \ldots , v_{n}\}$ be a basis of $\mathfrak{g}$,
  and $G$ be the simply connected Lie group with the Lie algebra $\mathfrak{g}$.
  If $\mathfrak{g}$ and $\mathcal{B}$ satisfy
\begin{equation}
  \label{eq:0803132915}
      \left\{
    \theta \in \mathrm{sym}(\mathfrak{g}) \relmiddle|
    \begin{array}{l}
      (i) \ \forall g \in \mathrm{Aut}(\mathfrak{g}) \cap O(n), \ \theta(g\cdot, g\cdot) = \theta(\cdot, \cdot),\\
      (ii) \  \forall A \in \R \oplus \mathrm{Der}(\mathfrak{g}),\  \sum_{i}\theta(v_{i},A v_{i}) = 0
    \end{array}
    \right\} = \{0\},
\end{equation}
then the left-invariant metric $\langle, \rangle$ on $G$ with orthonormal frame $\mathcal{B}$ is maximal.
\end{prop}

To find a Lie algebra with a basis which satisfies  (\ref{eq:0803132915}),
let us introduce the following notion:

\begin{defi}
  \label{defi:0728155920}
  Let $\mathfrak{g}$ be an $n$-dimensional Lie algebra.
  A basis $\mathcal{B} \subset \mathfrak{g}$ is called a \textit{$2$-reversible basis} if
  for any $v, w \in \mathcal{B}$ with $v \neq w$, there exists $g \in \mathrm{Aut}(\mathfrak{g}) \cap O(n)$
  and $a \in \{-1,1\}$ such that $gv = av$ and $gw = -aw$.
\end{defi}

For a $2$-reversible basis $\mathcal{B}$ on $\mathfrak{g}$, and
an $\mathrm{Aut}(\mathfrak{g}) \cap \mathrm{O}(n)$-invariant bilinear form $\theta : \mathfrak{g} \times \mathfrak{g} \to \R$,
  one has
    \begin{equation*}
    \theta(v,w) = \theta(gv,gw) = \theta(av,-aw) = -\theta(v,w) \quad (v,w \in \mathcal{B}, \ v \neq w).
  \end{equation*}
  This yields that 
\begin{lemm}
  \label{lemm:0721175133}
For a 2-reversible basis $\mathcal{B} \subset \mathfrak{g}$,
any $\mathrm{Aut}(\mathfrak{g}) \cap O(n)$-invariant bilinear form $\theta : \mathfrak{g} \times \mathfrak{g} \to \R$ is diagonal with respect to $\mathcal{B}$.
\end{lemm}

For examples, the Ricci tensor $\mathrm{Ric} : \mathfrak{g} \times \mathfrak{g} \to \R$ is $\mathrm{Aut}(\mathfrak{g}) \cap O(n)$-invariant
bilinear form.
Hence a $2$-reversible basis $\mathcal{B}$ is a  special case of a Ricci diagonal basis.

\subsection{almost abelian Lie groups with diagonal extension}
\label{20220802165433}

In this subsection, we prove Theorem~\ref{theo:1012173304}.
For $(w_{2}, w_{3}, \ldots , w_{n}) \in \R^{n-1}$, we consider a Lie group structure $\ast$ on $\R^{n}$ by
\begin{equation*}
  (x_{1}, x_{2}, \ldots,  x_{n}) \ast (y_{1}, y_{2}, \ldots, y_{n}) := (x_{1} + y_{1}, x_{2} + e^{w_{2}x_{1}}y_{2}, \ldots , x_{n} + e^{w_{n}x_{1}}y_{n}).
\end{equation*}
Namely, $(\R^{n}, \ast)$ is a semi-direct product $\R \ltimes \R^{n-1}$ whose action of $\R$ on $\R^{n-1}$ is given by
\begin{equation*}
  \R \to \mathrm{GL}(n-1,\R), \qquad t \mapsto \mathrm{diag}(e^{w_{2}t},e^{w_{3}t},\ldots , e^{w_{n}t}).
\end{equation*}
% Recall that Lie groups with the form $\R \ltimes \R^{n-1}$ are called \textit{almost abelian Lie groups}.
Lie groups with codimension one abelian normal subgroups such as $(\R^{n}, \ast)$ above are called \textit{almost abelian Lie groups}.

Define vector fields $v_{1}, v_{2}, \ldots , v_{n}$ on $\R^{n} := \{(x_{1}, x_{2}, \ldots , x_{n}) \mid x_{i} \in \R\}$ by%the Lie group $(\R^{n},\ast)$ by
\begin{equation}
  \label{eq:0810133022}
  v_{1} := \frac{\partial}{\partial x_{1}}, \quad v_{2} := e^{w_{2}x_{1}}\frac{\partial}{\partial x_{2}}, \quad \ldots,  \quad  v_{n} := e^{w_{n}x_{1}} \frac{\partial}{\partial x_{n}}.
\end{equation}
One can see that the vector fields $v_{1}, v_{2}, \ldots , v_{n}$ are left-invariant with respect to the Lie group structure $\ast$.
Note that $\mathcal{B}_{w} := \{v_{1}, v_{2}, \ldots , v_{n}\}$ is a basis of the Lie algebra of $(\R^{n},\ast)$.
The non-zero bracket relations of the Lie algebra $\mathfrak{s}_{w} := \mathrm{span}\{v_{1}, v_{2},\ldots , v_{n}\}$ are given by
\begin{equation}
  \label{eq:0902165034}
  [v_{1},v_{2}] = w_{2}v_{2}, \quad  [v_{1},v_{3}] = w_{3}v_{3}, \quad \ldots, \quad [v_{1}, v_{n}] = w_{n}v_{n}.
\end{equation}

For $w \in \R^{n-1}$,
denote by $g_{w} \in \mathfrak{M}(\R^{n})$ the left-invariant metric on $(\R^{n},\ast)$ with the orthonormal frame $\mathcal{B}_{w}$.
The metric $g_{w}$ is explicitly given by
\begin{equation}
  \label{eq:0902162104}
      g_{w} := (dx_{1})^{2} + e^{-2w_{2}x_{1}} (dx_{2})^{2} + \cdots  + e^{-2w_{n}x_{1}} (dx_{n})^{2},
\end{equation}
which is appeared in Theorem~\ref{theo:1012173304}.
By the result of Lauret in \cite{MR2770554},
one can easily see that $g_{w}$ is a left-invariant Ricci soliton on the simply connected solvable Lie group $(\R^{n},\ast)$.
Now we are in the position to prove Theorem~\ref{theo:1012173304}

% More generally, one has the following:

% \begin{prop}
%   \label{prop:0902161927}
%   For any $w := (w_{2},w_{3}, \ldots , w_{n}) \in \R^{n-1}$,
%   the Riemannian metric $g_{w}$ given by (\ref{eq:0902162104}) is a maximal metric on $\R^{n}$.
%   If $w_{i} \neq w_{j}$ for some $i,j \in \{2,3,\ldots , n\}$, then $g_{w}$ is not isotropy irreducible.
% \end{prop}

\begin{proof}[Proof of Theorem~\ref{theo:1012173304}]
  \label{proof:0902163643}
  Take any $w := (w_{2},w_{3}, \ldots , w_{n})$.
  We firstly show that the basis $\mathcal{B}_{w}\subset \mathfrak{s}_{w}$
  is a $2$-reversible basis.
  Take any $v_{i}, v_{j} \in \mathcal{B}_{w}$ with $i < j$.
   Let $\varphi \in \mathrm{GL}(n,\R)$ be a diagonal matrix whose $(j,j)$-element is $-1$ and the others are $1$.
   Note that $j > 1$.
  By the bracket relations (\ref{eq:0902165034}), one can see that $\varphi \in \mathrm{Aut}(\mathfrak{s}_{w}) \cap O(n)$, and one has
  $\varphi v_{i} = v_{i}$ and $\varphi v_{j} = -v_{j}$.

  Now we have only to show that
  the Lie algebra $\mathfrak{s}_{w}$ with the basis $\mathcal{B}_{w}$ satisfies
  the condition (\ref{eq:0803132915}) in Proposition~\ref{prop:0803132739}.  
  Take any $\theta \in \mathrm{sym}(\mathfrak{s}_{w})$ which satisfies (i) and (ii) in (\ref{eq:0803132915}).
  By Lemma~\ref{lemm:0721175133},
  $\theta$ is diagonal with respect to $\mathcal{B}_{w}$.
  Hence we have only to show that $\theta(v_{i}, v_{i}) = 0$ for each $i \in \{1, 2,\ldots ,n\}$.
  Take any $i \in \{1, 2,\ldots ,n\}$.
  By the bracket relations (\ref{eq:0902165034}), one can see that
  \begin{equation*}
    E_{22}, E_{33}, \ldots , E_{nn} \in \mathrm{Der}(\mathfrak{s}_{w}),
  \end{equation*}
  where $E_{pq} \in \mathfrak{gl}(\mathfrak{s}_{w})$ is
  a matrix whose $(p,q)$-entry is $1$, and the others are $0$.
  Note that $E_{11} = \mathrm{id}_{\mathfrak{s}_{w}} - E_{22} - E_{33} - \cdots - E_{nn} \in \R \oplus \mathrm{Der}(\mathfrak{s}_{w})$.
  This yields that $E_{ii} \in \R \oplus \mathrm{Der}(\mathfrak{s}_{w})$.
  By the condition (ii), one has
  \begin{equation*}
    0 = \sum_{j}\theta(v_{j},E_{ii}v_{j}) = \theta(v_{i}, v_{i}).
  \end{equation*}
  This concludes that $\theta = 0$.

  We show the last assertion.
  Denote by $\alpha_{w} := w_{2} + w_{3} + \cdots + w_{n}$.  
  Then one can see that the Ricci tensor of $g_{w}$ is given by
  \begin{equation}
    \label{eq:0721145841}
        \mathrm{Ric}(g_{w}) = \mathrm{diag}(-|w|^{2}, -w_{2}\alpha_{w},-w_{3}\alpha_{w}, \ldots , -w_{n}\alpha_{w}).
  \end{equation}
  This concludes that if $w_{i} \neq w_{j}$ for some $i,j$ then $g_{w}$ is not isotropy irreducible since $g_{w}$ is not Einstein.
\end{proof}

\begin{rem}
  \label{rem:0902164554}
  The equation (\ref{eq:0721145841}) yields that, if $n \geq 3$, $g_{w}$ gives continuous family of maximal metrics such that $g_{w} \nsim g_{w^{\prime}}$ for generic $(w, w^{\prime}) \in \R^{n-1} \times \R^{n-1}$.
\end{rem}

\subsection{nilpotent Lie algebras attached with edge transitive graphs}
\label{20220802165049}

In this subsection, we prove Theorem~\ref{theo:1207153743}.
Firstly, let us give a review on the $2$-step nilpotent Lie algebras attached to graphs, introduced in \cite{MR2140439}.
Let $\mathcal{V}$ be a finite set, and denote by $\mathfrak{P}_{2}^{\mathcal{V}} := \{e \subset \mathcal{V} \mid \sharp e = 2\}$.
Then, for a subset $\mathcal{E} \subset \mathfrak{P}_{2}^{\mathcal{V}}$, the pair $\mathcal{G} := (\mathcal{V},\mathcal{E})$ is called a \textit{(finite) simple graph}
with vertices $\mathcal{V}$ and edges $\mathcal{E}$.
For a simple graph $\mathcal{G} = (\mathcal{V},\mathcal{E})$, a map $d : \mathcal{E} \to \mathcal{V}$ is called a \textit{direction} of $\mathcal{G}$ if $d(e) \in e$
for all $e \in \mathcal{E}$.
Namely, a direction $d$ determines a starting point for each edge.
Also, denote by $d^{\ast} : \mathcal{E} \to \mathcal{V}$ the opposite direction of $d$ (\textit{i.e.} $\{d(e), d^{\ast}(e)\} = e$ for each $e \in \mathcal{E}$).
A simple graph $\mathcal{G}$ with a direction $\mathcal{G}_{d} := (\mathcal{G},d)$ is called a \textit{directed simple graph}. 

\begin{defi}
  \label{defi:0728145236}
For a directed simple graph $\mathcal{G}_{d} = (\mathcal{G},d)$ with vertices $\mathcal{V}$ and edges $\mathcal{E}$,
one can define a $2$-step nilpotent Lie algebra structure $[,]$ on $\mathrm{span}_{\R}(\mathcal{V} \cup \mathcal{E})$ as follows:
\begin{equation*}
  [d(e),d^{\ast}(e)] := e, \ [e, \cdot] := 0 \ \  (e \in \mathcal{E}), \quad [v,w] := 0 \ \  (\{v,w\} \notin \mathcal{E}).
\end{equation*}
We call the Lie algebra $\mathfrak{n}_{\mathcal{G}_{d}} := (\mathrm{span}_{\R}(\mathcal{V} \cup \mathcal{E}),[,])$
the \textit{$2$-step nilpotent Lie algebra attached with $\mathcal{G}$}.
\end{defi}

Note that $\mathfrak{n}_{\mathcal{G}_{d}}$ is abelian if and only if $\mathcal{E} = \emptyset$.
In the latter argument, we always assume that $\mathcal{E} \neq \emptyset$.

\begin{prop}
  \label{prop:0728145446}
  For any directed graph $\mathcal{G}_{d} = (\mathcal{V},\mathcal{E},d)$,
  the basis $\mathcal{V} \cup \mathcal{E} \subset \mathfrak{n}_{\mathcal{G}_{d}}$ is a $2$-reversible basis.
\end{prop}

\begin{proof}
  \label{proof:0712153240}

  For $v \in \mathcal{V}$,
  define $r_{v} \in \mathrm{Aut}(\mathfrak{n}_{\mathcal{G}_{d}}) \cap O(p+q)$ by
  \begin{equation*}
    r_{v}(v) := -v, \quad r_{v}(v^{\prime}) := v^{\prime} \ (v^{\prime} \in \mathcal{V}\setminus \{v\}), \quad
    r_{v}(e) := \left\{
      \begin{array}{ll}
        -e & (v \in e),\\
        e & (v \notin e).
      \end{array}
    \right.
  \end{equation*}
  We show that $\{r_{v} \mid v \in \mathcal{V}\}$ reflects the basis $\mathcal{V} \cup \mathcal{E}$.

  Firstly, for $v ,v^{\prime} \in \mathcal{V}$ with $v \neq v^{\prime}$,
  one has $r_{v}(v) = -v$ and $r_{v}(v^{\prime}) = v^{\prime}$.
  
  Secondly, for $v \in \mathcal{V}$, and $e \in \mathcal{E}$,
  put $v^{\prime} \in e \setminus \{v\}$.
  Then one has $r_{v^{\prime}}(v) = v$ and $r_{v^{\prime}}(e) = -e$.

  Lastly, for any $e, e^{\prime} \in \mathcal{E}$ with $e \neq e^{\prime}$,
  there exists a vertex $v \in e\setminus e^{\prime}$. 
  Then one has $r_{v}(e) = -e$ and $r_{v}(e^{\prime}) = e^{\prime}$.
\end{proof}

The following lemma asserts that invariant normal vectors $\theta \in (T_{\langle, \rangle_{\mathcal{G}}}\R_{>0}\mathrm{Aut}(\mathfrak{n}_{\mathcal{G}_{d}}).\langle, \rangle_{\mathcal{G}})^{\perp}$
are completely determined by the diagonal parts $\theta(e,e)$ along the edges $e \in \mathcal{E}$.

\begin{lemm}
  \label{lemm:0810134701}
  Let $\mathcal{G}_{d} = (\mathcal{G},d)$ be a directed simple graph.
  Take any $\theta \in \mathrm{sym}(\mathfrak{n}_{\mathcal{G}_{d}})$ which satisfies the condition (ii) in (\ref{eq:0803132915}).
  Then one has
  \begin{enumerate}
    [$(1)$]
  \item $\theta(v,v) = -\sum_{e \ni v}\theta(e,e)$ for all $v \in \mathcal{V}$,
  \item $\sum_{e \in \mathcal{E}}\theta(e,e) = 0$,    
  \end{enumerate}
\end{lemm}

\begin{proof}
  \label{proof:0810134848}
We firstly prove (1).
Take any $v \in \mathcal{V}$.
Define $D_{v} \in \mathrm{Der}(\mathfrak{n}_{\mathcal{G}_{d}})$ by
  \begin{equation*}
    D_{v}(v^{\prime}) :=
    \begin{cases}
      v \quad (v^{\prime} = v)\\
      0 \quad (v^{\prime} \neq v)
    \end{cases}
    \ (v^{\prime} \in \mathcal{V}),
    \quad
    D_{v}(e) :=
    \begin{cases}
      e \quad (v \in e)\\
      0 \quad (v \notin e)
    \end{cases}
    \ (e \in \mathcal{E}).
  \end{equation*}
  Since $\theta$ is a normal vector at $\langle , \rangle_{\mathcal{G}}$, one has $\sum_{x \in \mathcal{V} \cup \mathcal{E}} \theta(x,D_{v}x) = 0$ (see~(\ref{eq:0803132915})).
  This yields that
  \begin{align*}
    \label{align*:0406104714}
    0 = \sum_{x \in \mathcal{V} \cup \mathcal{E}} \theta(x,D_{v}x)
      = \sum_{v^{\prime} \in \mathcal{V}} \theta(v^{\prime},D_{v}v^{\prime}) + \sum_{e \in \mathcal{E}} \theta(e,D_{v}e)
      = \theta(v,v) + \sum_{e \ni v} \theta(e,e).
  \end{align*}
  This concludes that $\theta(v,v) = - \sum_{e \ni v} \theta(e,e)$.

  We secondly prove the assertion (2).
  Define $D \in \mathrm{Der}(\mathfrak{n}_{\mathcal{G}_{d}})$ by
  \begin{equation*}
    D(v) = v, \quad D(e) = 2e \qquad (v \in \mathcal{V}, \ e \in \mathcal{E}).
  \end{equation*}
  Then one has $D - I \in \R \oplus \mathrm{Der}(\mathfrak{n}_{\mathcal{G}_{d}})$. Note that
    \begin{equation*}
    (D-I)(v) = 0, \quad (D-I)(e) = e \qquad (v \in \mathcal{V}, \ e \in \mathcal{E}).
  \end{equation*}
  This yields that
  \begin{equation*}
    0 = \sum_{x \in \mathcal{V} \cup \mathcal{E}} \theta(x,(D - I)x) =  \sum_{e \in \mathcal{E}}\theta(e,e),
  \end{equation*}
  which completes the proof.
\end{proof}

For two finite simple graphs $\mathcal{G} = (\mathcal{V},\mathcal{E})$ and $\mathcal{G}^{\prime} = (\mathcal{V}^{\prime},\mathcal{E}^{\prime})$,
the bijective map $\sigma : \mathcal{V} \cup \mathcal{E} \to \mathcal{V}^{\prime} \cup \mathcal{E}^{\prime}$ is called an \textit{isomorphism}
between $\mathcal{G}$ and $\mathcal{G}^{\prime}$ if $\sigma$ satisfies $\sigma(\mathcal{V}) = \mathcal{V}^{\prime}$ and $\sigma(\{v,w\}) = \{\sigma(v),\sigma(w)\}$ for each $\{v,w\} \in \mathcal{E}$.
Denote by $\mathrm{Aut}(\mathcal{G})$ the group of self-isomorphisms of the graph $\mathcal{G}$.

For each $\sigma \in \mathrm{Aut}(\mathcal{G})$, one can define 
  $\tilde{\sigma} \in \mathrm{Aut}(\mathfrak{n}_{\mathcal{G}_{d}})$
  as follows:
  \begin{equation}
    \label{eq:0712225327}
    \tilde{\sigma}(v) := \sigma(v), \quad \tilde{\sigma}(e) :=
  \left\{
  \begin{array}{ll}
    \sigma(e) & (\mbox{if $\sigma(d(e)) = d(\sigma(e))$}), \\
    -\sigma(e) & (\mbox{if $\sigma(d(e)) \neq d(\sigma(e))$}),
  \end{array}
  \right.
  \end{equation}
where $v \in \mathcal{V}$, and $e \in \mathcal{E}$.
Then one has $\tilde{\sigma} \in \mathrm{Aut}(\mathfrak{n}_{\mathcal{G}_{d}}) \cap \mathrm{O}(\langle , \rangle_{\mathcal{G}})$.
Namely, the automorphism group  $\mathrm{Aut}(\mathcal{G})$ of the ``undirected'' graph $\mathcal{G}$ acts on
the normal space
$(T_{\langle, \rangle_{\mathcal{G}}}\R_{>0}\mathrm{Aut}(\mathfrak{n}_{\mathcal{G}_{d}}).\langle, \rangle_{\mathcal{G}})^{\perp}$
by the slice representation.

A simple graph $\mathcal{G}$ is called an edge-transitive graph if the automorphism group $\mathrm{Aut}(\mathcal{G})$ acts on $\mathcal{E}$ transitively.
For more informations about edge-transitive graphs, for examples, see \cite{MR4041267}.
It has been known that the left-invariant metric $\langle, \rangle_{\mathcal{G}}$ is a Ricci soliton if and only if the graph $\mathcal{G}$ is positive (\cite{MR2785768}).
One can easily see that edge-transitive graphs are positive, and hence $\langle, \rangle_{\mathcal{G}}$ for edge-transitive $\mathcal{G}$ are Ricci soliton.

For a directed simple graph $\mathcal{G}_{d}$, denote by $N_{\mathcal{G}_{d}}$ the simply connected nilpotent Lie group with Lie algebra $\mathfrak{n}_{\mathcal{G}_{d}}$.
Note that, if $\mathcal{G} = (\mathcal{V}, \mathcal{E})$ has $p$ vertices and $q$ edges, then $N_{\mathcal{G}_{d}}$ is diffeomorphic to $\R^{p+q}$ via the basis $\mathcal{V} \cup \mathcal{E}$.
Now we are in the position to prove Theorem~\ref{theo:1207153743}.
We show the following precise version of the Theorem~\ref{theo:1207153743}.

\begin{theo}
  \label{theo:0810135716}
  Let $\mathcal{G} = (\mathcal{V},\mathcal{E})$ be an edge-transitive graph.
  Then for any direction $d : \mathcal{E} \to \mathcal{V}$,
  the left-invariant metric $\langle, \rangle \in \mathfrak{M}(N_{\mathcal{G}_{d}})$ with the orthonormal frame $\mathcal{B} := \mathcal{V} \cup \mathcal{E}$ is maximal.
  The metric $\langle, \rangle$ is isotropy irreducible if and only if $\mathcal{E} = \emptyset$.
\end{theo}

\begin{proof}
  \label{proof:0712173323}
  Take any $\theta \in \mathrm{sym}(\mathfrak{n}_{\mathcal{G}_{d}})$ which satisfies
  \begin{align*}
    \label{align*:0616143618}
    \forall g \in \mathrm{Aut}(\mathfrak{g}) \cap O(n), \ \theta(g\cdot, g\cdot) = \theta(\cdot, \cdot),\
\forall A \in \R \oplus \mathrm{Der}(\mathfrak{g}),\  \sum_{i}\theta(v_{i},A v_{i}) = 0.
  \end{align*}
  By Proposition~\ref{prop:0803132739}, we have only to show that $\theta = 0$.
  By Lemma~\ref{lemm:0721175133},
  $\theta$ is diagonal with respect to $\mathcal{B}$.
  Also, by Lemma~\ref{lemm:0810134701}, the diagonal parts of $\theta$ are completely controlled by the edge parts $\theta(e,e)$.
  Hence we have only to show that $\theta(e,e) = 0$ for all $e \in \mathcal{E}$.

  Assume that $\mathcal{G}$ is edge-transitive.
  For any $e,e^{\prime} \in \mathcal{E}$, there exists $\sigma \in \mathrm{Aut}(\mathcal{G})$ such that $\sigma(e) = e^{\prime}$,
  and hence
  \begin{equation*}
    \theta(e^{\prime}, e^{\prime}) = \theta(\sigma(e),\sigma(e)) = \theta(\pm \tilde{\sigma}(e),\pm \tilde{\sigma}(e)) = \theta(e,e).
  \end{equation*}
  This concludes that the function $e \mapsto \theta(e,e)$ is constant.
  On the other hand, $\sum_{e^{} \in \mathcal{E}} \theta(e,e) = 0$ by Lemma~\ref{lemm:0810134701}.
  Therefore, for each $e \in \mathcal{E}$,
  \begin{equation*}
    0 = \sum_{e \in \mathcal{E}}\theta(e,e) = \sharp \mathcal{E} \cdot \theta(e,e),
  \end{equation*}
  and hence $\theta(e,e) = 0$.

  We have shown the last assertion. See Remark~\ref{rem:0714141634}.
\end{proof}

For a directed graph $\mathcal{G}_{d}$ with $\mathcal{G} = (\mathcal{V},\mathcal{E})$, denote by $\langle, \rangle_{\mathcal{G}_{d}}$ the left-invariant metric on $N_{\mathcal{G}_{d}}$ with the orthonormal frame $\mathcal{V} \cup \mathcal{E}$.
The above arguments yield that if given an edge-transitive graph $\mathcal{G}$ with $p$ vertices and $q$ edges, and a direction $d$, one can obtain maximal metrics  $\langle, \rangle_{\mathcal{G}_{d}}$ on $\R^{p+q}$.
On the other hand, the following states that
directions $d$ do not affect on the resulting left-invariant metric $\langle, \rangle_{\mathcal{G}_{d}}$.

\begin{theo}[\cite{MR3343346, MR661539}]
  \label{theo:0904000546}
  Let $\mathcal{G}_{d}$ and $\mathcal{G}_{d^{\prime}}^{\prime}$ be two directed simple graphs.
  Then the corresponding two simply connected Riemannian Lie groups $(N_{\mathcal{G}_{d}},\langle, \rangle_{\mathcal{G}_{d}})$ and $(N_{\mathcal{G}^{\prime}_{d^{\prime}}},\langle, \rangle_{\mathcal{G}^{\prime}_{d^{\prime}}})$ are isometric if and only if $\mathcal{G}$ and $\mathcal{G}^{\prime}$ are isomorphic as undirected graphs.
\end{theo}

Denote by $\mathfrak{G}_{p,q} := \{[\mathcal{G}]\mid  \mathcal{G} = (\mathcal{V}, \mathcal{E})\mbox{ : simple graph}, \ \sharp \mathcal{V} = p, \ \sharp \mathcal{E} = q\}$,
where $[\mathcal{G}]$ means the isomorphism class of $\mathcal{G}$.
% Recall that the simply connected Lie group $N_{\mathcal{G}_{d}}$  with the Lie algebra $\mathfrak{n}_{\mathcal{G}_{d}}$
% is naturally diffeomorphic to the Euclidean space through the basis $\mathcal{V} \cup \mathcal{E}$.
By Theorem~\ref{theo:0904000546}, the map
\begin{equation*}
  \mathfrak{G}_{p,q} \to \mathfrak{M}(\R^{p+q})/_{\! \sim}, \  [\mathcal{G}] \mapsto [\langle, \rangle_{\mathcal{G}_{d}}] 
\end{equation*}
is well-defined and injective.
In conclusion, we obtain  the following corollary for Theorem~\ref{theo:0810135716}
which states
that one can construct as many maximal metrics on the Euclidean spaces as many edge-transitive graphs,
hence there are plenty of maximal metrics on the Euclidean spaces.
\begin{cor}
  \label{cor:0712175939}
  Denote by $\mathfrak{G}_{p,q}^{edge} := \{[\mathcal{G}] \in \mathfrak{G}_{p,q} \mid \mathcal{G} \mbox{ : edge-transitive}\}$.
  Then the injective map $\mathfrak{G}_{p,q}^{edge} \to \mathfrak{M}(\R^{p+q})/_{\! \sim}$ maps each $[\mathcal{G}]$ to maximal elements $[\langle, \rangle_{\mathcal{G}_{d}}]$ in $\mathfrak{M}(\R^{p+q})/_{\! \sim}$.
  
\end{cor}

For examples, the following three edge-transitive graphs give three different maximal metrics on $\R^{10}$:

\begin{figure}[htbp]
  \begin{minipage}[b]{0.3\linewidth}
    \centering
    \includegraphics[bb=0 0 587 587, scale=0.2]{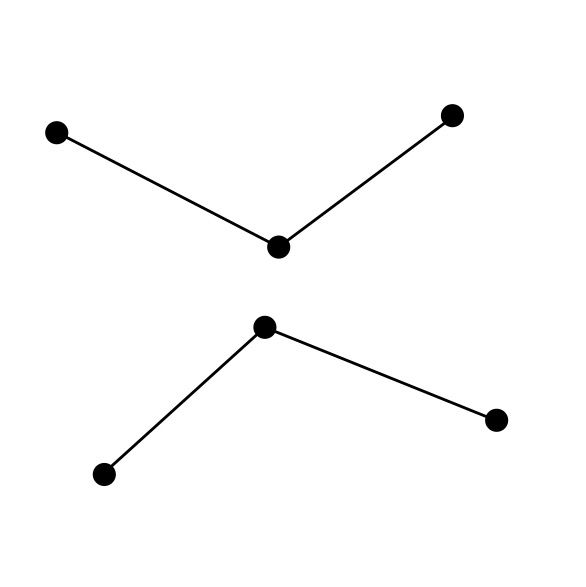}
    \subcaption{vertex: 6, edge: 4}
  \end{minipage}
  \begin{minipage}[b]{0.3\linewidth}
    \centering
    \includegraphics[bb=0 0 587 587, scale=0.2]{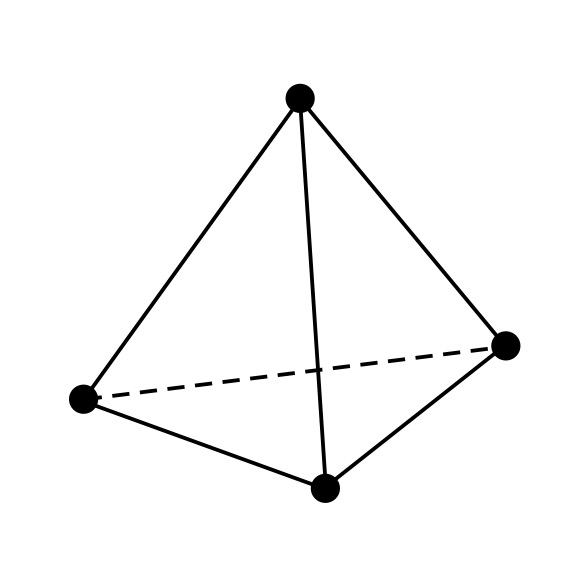}
    \subcaption{vertex: 4, edge: 6}
  \end{minipage}
    \begin{minipage}[b]{0.3\linewidth}
    \centering
    \includegraphics[bb=0 0 587 587, scale=0.2]{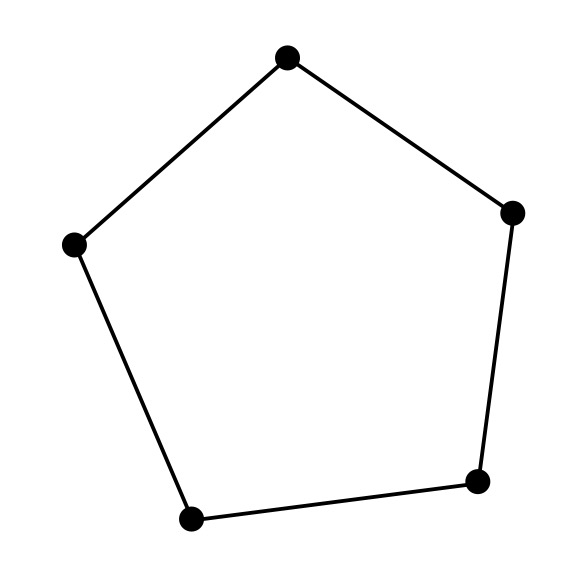}
    \subcaption{vertex: 5, edge: 5}
  \end{minipage}
\end{figure}

% =====================SECTION ENDS HERE=======================

%\bibliographystyle{plain}

%\bibliography{irrefinable}

\end{document}